\documentclass[10pt,a4paper,reqno]{amsart} %xcolor=dvipsnames
\usepackage[utf8]{inputenc}
\usepackage{a4wide}
\pdfoutput=1

\usepackage[english]{babel}
\usepackage[T1]{fontenc}
\usepackage{amsmath, mathtools}
\usepackage{amsfonts}
\usepackage{amssymb}
\usepackage{amsthm}
\usepackage[dvipsnames]{xcolor}
\usepackage{tikz}
\usetikzlibrary{positioning}
\usetikzlibrary{intersections}
\usetikzlibrary{calc}
\usepackage{csquotes}
\usepackage{MnSymbol}
\usepackage{bm}

\usepackage[backend = biber, style = alphabetic, giveninits=true, isbn = false, url = false, eprint = false, maxbibnames=99]{biblatex} 

\usepackage{hyperref}
\usepackage{cleveref}
\usepackage{xurl}
\hypersetup{breaklinks=true}

\newtheorem{theorem}{Theorem}[section]
\newtheorem{prop}[theorem]{Proposition}
\newtheorem{lemma}[theorem]{Lemma}
\newtheorem{cor}[theorem]{Corollary}

\theoremstyle{definition}

\newtheorem{example}[theorem]{Example}
\newtheorem{remark}[theorem]{Remark}

\numberwithin{equation}{section}

\newcommand{\PP}{\mathbb{P}}

\newcommand{\ZZ}{\mathbb{Z}}
\newcommand{\QQ}{\mathbb{Q}}
\newcommand{\RR}{\mathbb{R}}
\newcommand{\CC}{\mathbb{C}}

\newcommand{\norminf}[1]{\lVert #1 \rVert_\infty}
\newcommand{\norm}[1]{\lVert #1 \rVert}
\newcommand{\ringofintegers}{\mathcal{O}_{K}}
\newcommand{\localringofintegers}{\mathcal{O}_{K,\pprime}}
\newcommand{\pprime}{\mathfrak{p}}
\newcommand{\pvalue}[1]{v_\pprime(#1)}
\newcommand{\absvalue}[1]{\lvert #1 \rvert} 
\newcommand{\pabsvalue}[1]{\lvert #1 \rvert_\pprime}
\newcommand{\idealnorm}[1]{\mathfrak{N}(#1)}
\newcommand{\vol}{\mathrm{vol}}
\newcommand{\rk}{\mathrm{rk}}

\newcommand{\mapmindesingmodel}{\pi}
\newcommand{\constant}{\mathcal{A}}

\DeclareMathOperator{\Pic}{Pic}
\DeclareMathOperator{\Spec}{Spec}

\title[Integral Points on a del Pezzo Surface]{Integral Points on a del Pezzo Surface over Imaginary Quadratic Fields}
\author{Judith Ortmann}
\date{24 July 2023}
\address{Institut f\"ur Algebra, Zahlentheorie und Diskrete Mathematik, Leibniz Universit\"at Hannover, Welfengarten 1, 30167 Hannover, Germany}
\email{ortmann@math.uni-hannover.de}

\NewBibliographyString{toappear}
\DefineBibliographyStrings{english}{
	toappear = {Preprint, to appear},
}

\renewbibmacro*{in:}{
	\iffieldundef{pubstate}
	{}
	{\printfield{pubstate}
		\setunit{\addspace}
		\clearfield{pubstate}}
	\printtext{
		\bibstring{in}\intitlepunct}}

\begin{document}
	
	\begin{abstract}
		We characterise integral points of bounded log-anticanonical height on a quartic del Pezzo surface of singularity type $\mathbf{A}_3$ over imaginary quadratic fields with respect to its singularity and its lines. Furthermore, we count these integral points of bounded height by using universal torsors and interpret the count geometrically to prove an analogue of Manin's conjecture for the set of integral points with respect to the singularity and to a line.
	\end{abstract}

	
	\setcounter{tocdepth}{1}
	\maketitle
	\tableofcontents
	\allowdisplaybreaks

	\section{Introduction}
	\medskip
	
	Manin's conjecture \cite{fmt89,batyrev_manin90} predicts the asymptotic behaviour of the number of rational points on Fano varieties. In recent years, it was proved for various classes of varieties, for example for toric varieties \cite{batyrev1998toric}, equivariant compactifications of vector groups \cite{chambert-loir_tschinkel_2002} and some smooth del Pezzo surfaces \cite{breteche02,breteche11}.
	The leading constant appearing in the asymptotic formulas was made explicit, and the conjecture was generalised to include the singular del Pezzo surface we are considering by Peyre \cite{pey95,pey03} and Batyrev and Tschinkel \cite{batyrev_tschinkel98}.
	
	Results for integral points analogous to Manin's conjecture are often more difficult to prove, and less is known. Chambert-Loir and Tschinkel \cite{chambertloir_tschinkel_2010} constructed a framework for a geometric interpretation of the density of integral points which was refined by Wilsch \cite{wilsch_2021b} and proven for partial equivariant compactifications of vector groups \cite{chambertloir_tschinkel_2012} and some del Pezzo surfaces \cite{derenthal_wilsch_21}, for example.
	
	We recall three major methods that have been applied to Manin's conjecture and its analogue on integral points. 
	The \emph{circle method} was used to prove results for rational and integral points on high-dimensional complete intersections over $\QQ$ \cite{bir62,pey95}, and Loughran \cite{loughran15} generalised this work for rational points to arbitrary number fields from the work of Skinner \cite{ski97}.
	Tschinkel et al.\ used \emph{harmonic analysis} to give asymptotic formulas, for example, for the number of rational points on toric varieties and equivariant compactifications of vector groups over arbitrary number fields \cite{batyrev1998toric,chambert-loir_tschinkel_2002}. By using the same method, Chambert-Loir and Tschinkel analysed integral points on partial equivariant compactifications of vector groups over arbitrary number fields \cite{chambertloir_tschinkel_2012}. Takloo-Bighash and Tschinkel (split case) as well as Chow (nonsplit case) used harmonic analysis to analyse integral points on partial bi-equivariant compactifications of semi-simple groups of adjoint type \cite{takloo_bighash_tschinkel_2013, chow_19}.
	Similar to the harmonic analysis approach, in the area of homogeneous dynamics there are results using ergodic theory on exploiting a group action of linear algebraic groups and their homogeneous spaces to study the number of lattice points on certain affine varieties, see for example \cite{eskin_mozes_shah_96, eskin_mcmullen_93}. These results can be specialised to the case of integral points over imaginary quadratic fields.
	
	The \emph{universal torsor method} is particularly used for many singular and some smooth del Pezzo surfaces. Over the rational numbers there are, for example, results for rational points by de la Bretèche and Browning \cite{breteche02,breteche11}. 
	Later the universal torsor method was extended from $\QQ$ to other number fields. Derenthal, Frei and Pieropan applied this method over imaginary quadratic fields \cite{derenthal14_compos, derenthal_frei_compos2, derenthal_frei15, Pieropan_2016}, starting with a singular quartic del Pezzo surface with an $\mathbf{A}_3$ singularity. These results were generalised to arbitrary number fields by Frei and Pieropan \cite{frei_pieropan_15}.
	Derenthal and Wilsch \cite{derenthal_wilsch_21} used the universal torsor method to give an asymptotic formula for the number of integral points on a singular quartic del Pezzo surface with an $\mathbf{A}_1$ and an $\mathbf{A}_3$ singularity over $\QQ$. 
	
	We observe that the universal torsor method has been used for integral points so far only over $\QQ$. The aim of this paper is to start the generalisation of this method for integral points to number fields beyond $\QQ$. A first natural step is the consideration of imaginary quadratic fields: here we have to deal with class number greater than one, but the group of units is still finite and we only have one archimedean place.
	As a first example, we study the asymptotic behaviour of integral points of bounded height over a singular del Pezzo surface $S$ of degree 4 with an $\mathbf{A}_3$ singularity over imaginary quadratic fields. This is the same del Pezzo surface that Derenthal and Frei considered while generalising the universal torsor method from $\QQ$ to imaginary quadratic fields for rational points, hence it seems to be a good starting point. 
	
	Our main result (\Cref{main_theorem}) is an asymptotic formula for the number of integral points of bounded height on this chosen del Pezzo surface $S$ with respect to two natural choices of boundaries: the singularity and a line.
	This asymptotic formula is of the shape
	\[cB(\log B)^{b-1}.\]
	We will describe the leading constant $c$ and the exponent $b$ later in this introduction. Further, we will see that this formula can be interpreted geometrically (\Cref{section:the expected asymptotic formula} and \Cref{section: The leading constant}). We will show that the leading constant $c$ consists of Tamagawa numbers and combinatorial constants.
	
	To the author's best knowledge, this is the first example of counting integral points beyond $\QQ$ outside the reach of the circle method and without exploiting a group action. We investigate some of the constructions for integral points in a new setting: we choose a variety without a given group action and with more complex geometry than hypersurfaces, which results in Picard groups of high rank.
	We hope that our results will help to better conceptually understand how integral points behave, as less is known in general.

	\subsection{The counting problem}
	Let $K$ be an imaginary quadratic field of arbitrary class number $h_K$, i.e.\ $K=\QQ(\sqrt{d})$ for a negative squarefree integer $d$. We denote by $\ringofintegers$ its ring of integers.
	Let \[\mathcal{C}=\{P_1 = \ringofintegers, P_2,\dots,P_{h_K}\}\] be a fixed system of integral representatives for the ideal classes of $K$.
	See \Cref{section_notation} for further standard notation that we will use already in the introduction.
	
	We consider the anticanonically embedded del Pezzo surface $S\subseteq\PP^4_K$ given by the equations
	\begin{equation}\label{eq_del_Pezzo}
		x_0x_1-x_2x_3 = x_0x_3 + x_1x_3 + x_2x_4 = 0.
	\end{equation}
	It contains exactly one singularity $Q=(0:0:0:0:1)$, and its type is $\mathbf{A}_3$. 
	Our goal is to count integral points of bounded log-anticanonical height on $S$. 
	
	We consider the integral model $\mathcal{S}\subseteq\PP_{\ringofintegers}^4$ of $S$ defined by the same equations over $\ringofintegers$. 
	The closure of every rational point $P\in S(K)$ is an integral point $\overline{P}\in\mathcal{S}(\ringofintegers)$. 
	On the projective variety $\mathcal{S}$, rational and integral points coincide. Hence, we choose an appropriate \emph{boundary} $\mathcal{Z}$ to consider integral points on $\mathcal{S}\setminus\mathcal{Z}$ to make the counting problem interesting.
	A general treatment of such boundaries for del Pezzo surfaces of low degree can be found in \cite[Theorem 10]{derenthal_wilsch_21}.
	We choose two eligible types of boundaries: the singularity and one of the lines of $S$. We start with the former. 
	
	Let $Z_1=Q$, $\mathcal{Z}_1=\overline{Z}_1$, and $\mathcal{U}_1 = \mathcal{S}\setminus \mathcal{Z}_1$. An integral point on $S\setminus Z_1$ is a rational point $\bm{x}\in S(K)$ such that  the corresponding integral point in $\mathcal{S}(\ringofintegers)$ does not meet the closure $\mathcal{Z}_1$ of $Q$ in $\mathcal{S}$. We recall that due to \cite[Section 1]{Schanuel_1979}, any integral or rational point on $S$ can be represented (uniquely up multiplication by units) by $(x_0,\ldots,x_4)\in\ringofintegers^5\setminus\{(0,\dots,0)\}$ satisfying the defining equation \eqref{eq_del_Pezzo} and 
	\begin{equation}\label{corresponding_gcd_ideal}
		x_0\ringofintegers+\dots+x_4\ringofintegers = P_j
	\end{equation} 
	for some $j=1,\ldots,h_K$. 
	A representative $\bm{x} = (x_0:\dots:x_4)$ of a point in $\mathcal{U}_1(\ringofintegers)$ with integral coordinates and \eqref{corresponding_gcd_ideal} satisfies $(x_0:\dots:x_4)\neq Q$ in the residue field $\ringofintegers/\pprime$ for all prime ideals $\pprime$. This means that $\bm{x}$ satisfies the \emph{integrality condition}
	\begin{equation}\label{integrality condition}
		x_0\ringofintegers+\dots+x_3\ringofintegers=P_j.
	\end{equation}
	
	Clearly, the set of integral points $\mathcal{U}_1(\ringofintegers)$ is infinite. Therefore, we consider integral points of bounded height and work with the following \emph{height function}: 
	\begin{equation}\label{height_function_general}
		H_1(\bm{x}) = \frac{\max\lbrace \norminf{x_0}, \norminf{x_1},  \norminf{x_2},  \norminf{x_3}\rbrace}{\idealnorm{x_0\ringofintegers+\cdots+x_3\ringofintegers}}.
	\end{equation}
	We will later see that this can be interpreted as a \emph{log-anticanonical height} on a minimal desingularisation of $S$.
	The number of integral points of bounded height is dominated by the integral points on the five lines 
	\begin{gather}\label{lines_L_i}
		L_1 = \lbrace x_0=x_1=x_2=0\rbrace,\, L_2 = \lbrace x_0=x_2=x_3=0\rbrace,\,L_3 = \lbrace x_0=x_3=x_4=0\rbrace,\,\\
		L_4 = \lbrace x_1=x_2=x_3=0\rbrace,~\text{and}~ L_5 = \lbrace x_1=x_3=x_4=0\rbrace.
	\end{gather}
	Hence, we count integral points on their complement \[V=S\setminus\{x_0x_3= 0\}\] in $S$,
	and we are interested in the asymptotic behaviour of
	\begin{equation}\label{def_N_i}
		N_1(B) = \#\lbrace \bm{x}\in \mathcal{U}_1(\ringofintegers)\cap V(K)\mid H_1(\bm{x})\le B\rbrace,    
	\end{equation}
	the number of integral points of bounded log-anticanonical height that are not contained in the lines, as the height bound $B$ tends to infinity.
	Explicitly, this is
	\begin{equation}\label{descr_N_i}
		N_1(B) = \frac{1}{\omega_K}\sum_{j=1}^{h_K}\#\lbrace (x_0,\ldots,x_4)\in\ringofintegers^5\mid x_0x_3\neq 0,\, \eqref{eq_del_Pezzo},\,\eqref{corresponding_gcd_ideal},\,\eqref{integrality condition},H_1(\bm{x})\le B\rbrace.
	\end{equation}
	
	As a second type of a boundary we choose the line $Z_2 = L_2$. Let $\mathcal{Z}_2 = \overline{Z}_2$ in $\mathcal{S}$, and $\mathcal{U}_2 = \mathcal{S}\setminus \mathcal{Z}_2$. Analogously to the first case, a point $\bm{x}=(x_0:\dots:x_4)\in S$ satisfying \eqref{corresponding_gcd_ideal} with $x_0,\ldots,x_4\in\ringofintegers$ lies in $\mathcal{U}_2(\ringofintegers)$ if and only if
	\begin{equation}\label{integrality_condition_2}
		x_0\ringofintegers + x_2\ringofintegers + x_3\ringofintegers = P_j. 
	\end{equation}
	We use the height
	\begin{equation}\label{height_function_2}
		H_2(\bm{x}) = \frac{\max\lbrace  \norminf{x_0},\norminf{x_2}, \norminf{x_3}\rbrace}{\idealnorm{x_0\ringofintegers + x_2\ringofintegers + x_3\ringofintegers}},
	\end{equation}
	which will again turn out to be log-anticanonical on the minimal desingularisation of $S$. Let $N_2(B)$ be defined analogously to \eqref{def_N_i} with $\mathcal{U}_1$ and $H_1$ replaced by $\mathcal{U}_2$ and $H_2$, respectively. It satisfies the description in \eqref{descr_N_i} with the integrality condition \eqref{integrality condition} replaced by \eqref{integrality_condition_2}, and $H_1$ by $H_2$.
	
	We prove the following asymptotic formulas for these counting problems:
	\begin{theorem}\label{main_theorem}
		As $B\rightarrow\infty$, we have 
		\begin{align*}
			N_1(B) &= \frac{\rho_K^3}{\sqrt{\lvert\Delta_K\rvert}^2}\frac{\pi^2}{18}\prod_\mathfrak{p} \left( 1-\frac{1}{\idealnorm{\pprime}}\right)^3\left(1+\frac{3}{\idealnorm{\pprime}}\right) 
			B(\log B)^4
			+ O(B(\log B)^3\log\log B),\quad\text{and}\\
			N_2(B) &= \frac{\rho_K^2}{\sqrt{\absvalue{\Delta_K}}^2} \frac{11\pi^2}{18}\prod_{\pprime}\left( 1-\frac{1}{\idealnorm{\pprime}}\right)^2\left(1+\frac{2}{\idealnorm{\pprime}}\right)B(\log B)^3
			+ O(B(\log B)^2\log\log B).
		\end{align*}
	\end{theorem}
	We obtain an analogous result over $\QQ$ (\Cref{theorem_over_Q}), which we will state and prove in \Cref{section: simplifications for the rational numbers}.
	
	\subsection{The expected asymptotic formula}\label{section:the expected asymptotic formula}
	Our asymptotic formulas for the number of integral points of bounded height should be interpreted on a minimal desingularisation $\pi\colon\widetilde{S}\rightarrow S$. Here, $\widetilde{S}$ is a \emph{weak del Pezzo surface}, that is, a smooth projective surface whose anticanonical bundle $\omega_{\widetilde{S}}^\vee$ is big and nef. Equivalently, weak del Pezzo surfaces are the smooth del Pezzo surfaces and the minimal desingularisations of del Pezzo surfaces with only ADE-singularities \cite{demazure80}. 
	Analogously to \cite{derenthal_wilsch_21}, we study a desingularisation $\widetilde{U}_i = \widetilde{S}\setminus D_i$ of $U_i$, where $D_i = \pi^{-1}(Z_i)$ is a reduced effective divisor with strict normal crossings, $i=1,2$. We use this to interpret the number of points on $\mathcal{U}_i = \mathcal{S}\setminus\mathcal{Z}_i$. When studying integral points, the log-anticanonical bundle $\omega_{\widetilde{S}}(D_i)^\vee$ replaces the anticanonical bundle. From this perspective, we can interpret \Cref{main_theorem} in the framework described in \cite{chambertloir_tschinkel_2010}.
	
	The minimal desingularisation $\widetilde{S}$ is obtained from $\PP_K^2$ by a chain of five blow-ups. We will see in \Cref{section: Passage to a universal torsor} that the same chain of blow-ups of $\PP^2_{\ringofintegers}$ results in an integral model $\mapmindesingmodel\colon\widetilde{\mathcal{S}}\rightarrow\mathcal{S}$. Then, $D_1$ is the divisor above $Q$.
	Let $\widetilde{U}_i,\,\widetilde{\mathcal{U}}_i$ be the complement of $D_i,\,\overline{D}_i$ in $\widetilde{S},\,\widetilde{\mathcal{S}}$, respectively, where $\overline{D}_i$ is the Zariski closure of $D_i$ in $\widetilde{\mathcal{S}}$. The complement $\widetilde{V}$ of all negative curves on $\widetilde{\mathcal{S}}$ is obtained as the preimage of the lines on $S$, that means, $\widetilde{V} = \pi^{-1}(V)$.
	
	We can reinterpret our counting problem on the minimal desingularisation as follows:
	\begin{equation*}
		N_i(B) = \#\lbrace \bm{x}\in \widetilde{\mathcal{U}}_i(\ringofintegers)\cap\widetilde{V}(K)\mid H_i(\pi(\bm{x}))\le B\rbrace.
	\end{equation*}
	In \Cref{height_function_on_desingularisation}, we will prove that $H_i\circ\mapmindesingmodel$ is a log-anticanonical height function on $\widetilde{\mathcal{U}}_i(\ringofintegers)\cap \widetilde{V}(K)$.
	
	For example as in \cite{chambertloir_tschinkel_2012,derenthal_wilsch_21}, we expect that
	\begin{equation}\label{N_i(B)_minimal_desing}
		N_i(B)= c_{i,\mathrm{fin}}c_{i,\infty} B(\log B)^{b_i-1}(1+o(1)),
	\end{equation}
	where the leading constant can be decomposed into a finite part $c_{i,\text{fin}}$ and an archimedean part $c_{i,\infty}$, with
	\begin{align}
		c_{i,\mathrm{fin}} &= \rho_K^{\rk(\Pic(\widetilde{U}_i))}\prod_\pprime\left(1-\frac{1}{\idealnorm{\pprime}}\right)^{\rk(\Pic(\widetilde{U}_i))}
		\tau_{(\widetilde{S},D_i),\pprime}(\widetilde{\mathcal{U}}_i(\localringofintegers)), \label{const_fin}\\
		c_{i,\infty} &= \frac{1}{\absvalue{\Delta_K}^{\dim(U_i)/2}}\sum_{A\in\mathcal{C}^{\max,0}(D)}\alpha_{i,A}\tau_{i,D_{A,\infty}}(D_{A}(K_\infty)),\label{const_infty}
	\end{align}
	and 
	\[b_i = \rk \Pic(\widetilde{U}_i) + \dim \mathcal{C}^\mathrm{an}_\CC(D_i)+1.\]
	Here, $\tau_{(\widetilde{S},D_i),\pprime}$ is a $\pprime$-adic Tamagawa measure, $\mathcal{C}^{\max,0}(D)$ denotes the set of faces $A$ of the (analytic) Clemens complex of maximal dimension, which correspond to the minimal strata $D_A$ of $D$. The constant $\alpha_{i,A}$ is some rational number and $\tau_{i,D_{A},\infty}$ is an archimedean Tamagawa measure. Moreover, $\dim \mathcal{C}_\CC^\mathrm{an}(D_i)+1$ is the maximal number of components of the boundary divisor $D_i$ having non-empty intersection (that means, that meet in the same point).
	For more details, we refer to \cite{derenthal_wilsch_21} and \Cref{section: The leading constant}, where we will describe and determine these constants precisely.
	
	We show in \Cref{section: The leading constant} that \Cref{main_theorem} coincides with the expectation \eqref{N_i(B)_minimal_desing}. In particular, it turns out that the product of the constants $c_{i,\text{fin}}$ and $c_{i,\infty}$ equals the constant computed in \Cref{main_theorem}.
	
	\subsection{Strategy of the paper}
	The paper is organised as follows. In \Cref{section: Passage to a universal torsor}, we deal with the parameterisation of the set of integral points on our surface by integral points on a universal torsor. We start by describing this universal torsor on the minimal desingularisation of $S$. The main difficulty here is to adapt Derenthal and Wilsch's method to the setting of class number greater than $1$. We can no longer choose our representatives for integral points on the model coprime, that means uniquely up to units. We will consider representatives of these integral points where the coordinates do not necessarily lie in $\ringofintegers$. Hence, it will be difficult to reduce them modulo all prime ideals in the ring of integers. This makes it harder to decide which representatives are integral points. 
	
	We can use \cite[Theorem 2.5]{frei_pieropan_15} to prove a first representation (\Cref{lemma_parameterisation_of_integer_points_via_universal_torsor}) for the set of integral points, in which ideals corresponding to the torsor variables that correspond to the boundary divisor have to coincide with the ring of integers $\ringofintegers$. This not necessarily implies that these torsor variables have to be units. 
	To obtain a similar result to \cite{derenthal_wilsch_21}, and to get a nicer representation of the set of integral points to make the counting problem easier, we choose different (but isomorphic) twists of our universal torsor. Then, we roughly obtain that the torsor variables corresponding to the boundary divisor must be units (\Cref{lemma_nicer_representation}).
	In fact, while the integral points in our first parameterisation are a subset of those parameterising the rational points in \cite{frei_pieropan_15}, we observe that this is no longer true in our second parameterisation (see also \Cref{example:integral_no_subset_of_rational}).
	
	Further, we show that our height functions $H_1$ and $H_2$ are log-anticanonical and give an explicit description of these by monomials in the Cox ring of log-anticanonical degree.
	This allows us to formulate an explicit counting problem on the universal torsor (\Cref{prop: N_i(B)}).
	
	In \Cref{section: summations}, we perform the summations to estimate the number of integral points on the universal torsor using analytic techniques. The first step is to approximate the sums over the torsor variables by integrals (\Cref{lemma_first_summation,lemma_remaining_summations} and \Cref{lemma:N_i(B)_with_V_0}). Most computations work similarly as in \cite{derenthal_wilsch_21}, hence we will be brief here. However, the transformation of the sum over the torsor variables into a sum over ideals is more complicated since one of the ideal classes is dependent on the others, which leads to an extra factor $h_K^{-1}$ (\Cref{lemma_transformation_eta_into_ideals}). The coprimality conditions lead to an Euler product of local densities, which agrees with  $c_{i,\text{fin}}$ up to a few constants depending on $K$ (\Cref{lemma_remaining_summations}). The missing constants appear in the remaining integral.
	To complete the proof of \Cref{main_theorem}, we need to transform the obtained integral into \[\frac{\pi^{\mathrm{rk}(\Pic(\widetilde{U}_i))}}{4}C_i\cdot B\left(\log B\right)^{\mathrm{rk}(\Pic(\widetilde{U}_i))},\] 
	where $C_i$ is the product of the volume of a polytope (which agrees with $\sum \alpha_{i,A}$) and a real density (which coincides with the archimedean Tamagawa numbers $\tau_{i,D_A,\infty}(D_{A}(\CC))$). This transformation works with a combination of the arguments in \cite{derenthal14_compos} and \cite{derenthal_wilsch_21}. We slightly change the integration area by producing negligible error terms (\Cref{lemma:change_the_integration_area_of_V_0} and \Cref{cor_V_0}), and then transform the complex integration variables into real ones by using polar coordinates (\Cref{def_C_i}).
	
	In \Cref{section: The leading constant}, we explicitly compute the expected leading constant discussed in \Cref{section:the expected asymptotic formula} and prove that \eqref{N_i(B)_minimal_desing} holds. 
	Finally, we sketch the proof of the analogue of \Cref{main_theorem} in the case of the rational numbers in \Cref{section: simplifications for the rational numbers}.

	\subsection{Notation}\label{section_notation}
	By $\Delta_K$ we denote the discriminant of $K$, by $R_K$ the regulator and $\omega_K$ denotes the number of roots of unity. Further, let\[\rho_K = \frac{2^{s_1}(2\pi)^{s_2}h_KR_K}{\omega_K\sqrt{\lvert\Delta_K\rvert}},\] 
	where $s_1$ is the number of real embeddings of $K$ and $s_2$ is the number of pairs of complex embeddings. For $K$ imaginary quadratic, we have $s_1=0$ and $s_2 =1$. We note that for imaginary quadratic fields $K$ and $K=\QQ$, it is always $R_K=1$.
	
	When we use Vinogradov's $\ll$-notation or Landau's $O$-notation, the implied constants may always depend on $K$. In cases where they may depend on other objects as well, we mention this, for example by writing $\ll_C$ or $O_C$ if the constant may depend on $C$.
	
	In addition, we denote by $\mathcal{I}_K$ the monoid of nonzero ideals of $\ringofintegers$. The symbol $\mathfrak{a}$ (respectively $\pprime$) always denotes an ideal (respectively nonzero prime ideal) of $\ringofintegers$, and $\pvalue{\mathfrak{a}}$ is the nonnegative integer such that $\pprime^{\pvalue{\mathfrak{a}}}\divides \mathfrak{a}$ and $\pprime^{\pvalue{\mathfrak{a}}+1} \nmid \mathfrak{a}$. We extend this in the usual way to fractional ideals (with $\pvalue{\lbrace0\rbrace}=\infty)$, and for $x\in K$, write $\pvalue{x}= \pvalue{x\ringofintegers}$ for the usual $\pprime$-adic exponential valuation. We denote by $\mathcal{O}_{K,\pprime}$ the ring of integers of the completion $K_\pprime$ of $K$ at $\pprime$.
	We equip the completions of $K$ with the norms $\norm{\cdot}_\omega$ such that
	\[\norm{x}_\omega = \absvalue{N_{K_\omega/\QQ_v}(x)}_v\]
	at a place $\omega$ lying above a place $v$ of $\QQ$ and such that $\absvalue{p}_p = 1/p$ on $\QQ_p$ and $\absvalue{\cdot}_\infty$ is the standard absolute value $\absvalue{\cdot}$ on $\RR$. In particular, we then have the convention $\norminf{\cdot} = \lvert \cdot \rvert^2$, where $\lvert\cdot\rvert$ is the usual complex absolute value.
	Lastly, for a divisor $D$ we write $\vert D\vert$ for the support of $D$.

	\subsection*{Acknowledgements}
	The author was partially supported by the Caroline Herschel Programme of Leibniz University Hannover, as well as a scholarship for a research stay abroad of the Graduiertenakademie of Leibniz University Hannover since part of the work was done while visiting the University of Bath. I wish to thank Daniel Loughran and the math departement for their hospitality at the University of Bath. Further, I am thankful to Marta Pieropan and Florian Wilsch for their useful remarks, suggestions for improvements and the fruitful discussions with them.

	\section{Passage to a universal torsor}\label{section: Passage to a universal torsor}
	Analogously to \cite{derenthal14_compos,frei_pieropan_15,derenthal_wilsch_21}, we use universal torsors to parameterise the set of integral points on $U_i\subseteq S$ by integral points on an affine hypersurface. We use the same notation and numbering as in \cite{derenthal14_compos}. 
	
	Let $\overline{K}$ be an algebraic closure of $K$, and let $\widetilde{S}_{\overline{K}}$ be the minimal desingularisation of $S_{\overline{K}}$ as in \cite{derenthal14_compos}. 
	The data in \cite[§\,3.4]{derenthal14_lms} shows that $\widetilde{S}_{\overline{K}}$ is obtained by a blowing-up of $\PP^2_{\overline{K}}$ in five points in almost general position with Picard group $\Pic(\widetilde{S}_{\overline{K}})$ isomorphic to $\ZZ^6$.
	The Cox ring of $\widetilde{S}_{\overline{K}}$ is the $\Pic(\widetilde{S}_{\overline{K}})$-graded $\overline{K}$-algebra 
	\[R_{\overline{K}} = \overline{K}[\eta_1,\ldots,\eta_9]/(\eta_1\eta_4^2\eta_7+\eta_3\eta_6^2\eta_8+\eta_5\eta_9),\]
	which is defined by nine generators and one homogeneous relation.
	For $i\in\lbrace1,\ldots,9\rbrace$, the generator $\eta_i$ has degree $[E_i]\in\Pic(\widetilde{S}_{\overline{K}})$, and the divisor classes $[E_i]$ are given as follows.
	Let $l_0,\ldots,l_5$ be the basis of $\Pic(\widetilde{S}_{\overline{K}})$ given in \cite{derenthal14_lms}. Then,  $l_0^2=1,l_i^2=-1$ for $1\le i\le5$, and $l_i. l_j=0$ for all $0\le i<j\le 5$ gives the intersection form. We have
	\begin{align}\label{degree_E_i}
		\begin{gathered}
			[E_1] = l_1 - l_4,\quad 
			[E_2] = l_0-l_1-l_2-l_3,\quad
			[E_3] = l_2-l_5,\\
			[E_4] = l_4,\quad
			[E_5] = l_3,\quad
			[E_6] = l_5,\quad
			[E_7] = l_0-l_1-l_4,\quad
			[E_8] = l_0-l_2-l_5,\\
			\text{and}\ [E_9] = l_0-l_3.
		\end{gathered}
	\end{align}
	We will see in this section that there is an ideal $J$ (see \eqref{polynomials f_i} for the construction) and a morphism $\rho$ such that $\rho\colon \overline{Y} = \Spec(R_{\overline{K}})\setminus V(J) \rightarrow{\widetilde{S}_{\overline{K}}}$ is a universal torsor. This gives us the correspondence $V(\eta_i) = \rho^{-1}(E_i)$.
	
	The extended Dynkin diagram in \Cref{fig:1-Dynkin-diagram} encodes the configuration of curves corresponding to generators of $\mathrm{Cox}(\widetilde{S}_{\overline{K}})$. There are $[E_j].[E_k]$ edges between the vertices corresponding to $E_j$ and $E_k$. We mark a vertex by a circle (respectively a box) when it corresponds to a ($-2$)-curve (respectively ($-1$)-curve). 
	
	Similarly as in \cite{derenthal_wilsch_21}, the sum of the ($-2$)-curves on $\widetilde{S}_{\overline{K}}$ above the singularity $Q$ on $S_{\overline{K}}$ corresponds to the divisor $D_1=E_1+E_2+E_3$. The $(-1)$-curves $E_4,E_5,E_6,E_7,\text{ and }  E_8$ are the strict transforms of the five lines $L_2,L_1,L_4,L_3,\text{ and }L_5$, respectively, which were defined in \eqref{lines_L_i}.
	Above the lines $L_2,L_3,L_4$ and $L_5$, respectively, lie the divisors
	\[D_2 = E_1+E_2+E_3+E_4,\ D_3 = E_7,\ D_4 = E_1+E_2+E_3+E_6\  \text{and}\ D_5 = E_8.\]
	As in \cite{derenthal_wilsch_21}, the preimage $\widetilde{V}\subset\widetilde{S}$ of $V$ is the complement of the negative curves $E_1,\ldots,E_8$.
	
	\begin{figure}[ht]
		\centering
		\begin{tikzpicture}[
			roundnode/.style={circle, draw=black, thick, minimum size=5mm},
			squarednode/.style={rectangle, draw=black, thick, minimum size=8mm},
			]
			%Nodes
			\node[squarednode]      (E_7)                              {$E_7$};
			\node[squarednode]      (E_4)               [right= 1.5cm of E_7]  {$E_4$};
			\node[roundnode]      (E_1)               [right= 1.5cm of E_4]  {$E_1$};
			\node[]      (E_9)               [below= 0.8cm of E_4]  {$E_9$};
			\node[squarednode]      (E_5)               [right= 1.5cm of E_9]  {$E_5$};
			\node[roundnode]      (E_2)               [right= 1.5cm of E_5]  {$E_2$};
			\node[squarednode]      (E_6)               [below= 0.8cm of E_9]  {$E_6$};
			\node[roundnode]      (E_3)               [right= 1.5cm of E_6]  {$E_3$};
			\node[squarednode]      (E_8)               [left= 1.5cm of E_6]  {$E_8$};
			
			%Lines
			\draw[-] (E_7.south) -- (E_8.north);
			\draw[-] (E_7.east) -- (E_4.west);
			\draw[-] (E_4.east) -- (E_1.west) node [midway, above] {$A_3$};
			\draw[-] (E_8.east) -- (E_6.west);
			\draw[-] (E_6.east) -- (E_3.west) node [midway, above] {$A_4$};
			\draw[-] (E_9.east) -- (E_5.west);
			\draw[-] (E_5.east) -- (E_2.west);
			\draw[-] (E_7.south east) -- (E_9.north west);
			\draw[-] (E_8.north east) -- (E_9.south west);
			\draw[-] (E_1.south east) -- (E_2.north west) node [midway, above] {$A_1$};
			\draw[-] (E_3.north east) -- (E_2.south west) node [midway, below] {$A_2$};
		\end{tikzpicture} 
		\caption{Configuration of the curves on $\widetilde{S}_{\overline{K}}$ and the faces $A_i$ of the Clemens complexes.}
		\label{fig:1-Dynkin-diagram}
	\end{figure}
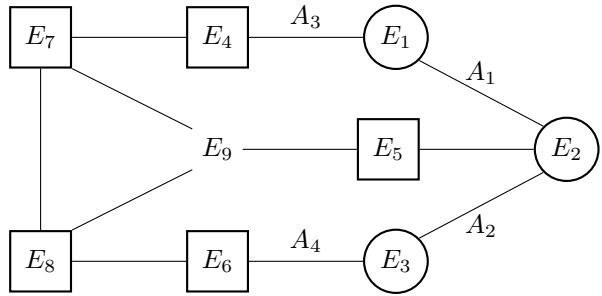

	\begin{remark}\label{remark: i=3 not countable}
		We recall that \cite[Theorem 10]{derenthal_wilsch_21} gives boundaries for del Pezzo surfaces of degree $d\le 4$ such that the minimal desingularisation together with the reduced effective boundary divisor $D$ is a weak del Pezzo pair, i.e.\ the log anticanonical bundle $\omega_{\widetilde{S}}(D)^\vee$ is big and nef.  For our counting problem, the possible boundaries are the singularity or one of the lines except $L_1$. Due to two symmetric cases ($L_2$ and $L_4$, and $L_3$ and $L_5$), there are three different types of boundaries that we can choose: the singularity $Q$, the line $L_2$, and the line $L_3$. 
		The line $L_1$ cannot be chosen as a boundary since the corresponding $(-1)$-curve on the minimal desingularisation of $S$ does not form a chain with the $(-2)$-curves corresponding to the singularity~$Q$. Hence, the log-anticanonical bundle $\omega_{\widetilde{S}}(D)^{\vee}$ of the corresponding divisor $D = E_1+E_2+E_3+E_5$ is not nef, as its intersection number with $E_2$ is $-1$.
		
		In this section, we parameterise integral points on a universal torsor in all three cases. Therefore, we briefly describe the setting for the third case, which is not mentioned yet.
		Analogously to the other two cases, we set $Z_3 = L_3$, let $\mathcal{Z}_3 = \overline{Z_3}$ in $\mathcal{S}$ and $\mathcal{U}_3 = \mathcal{S}\setminus \mathcal{Z}_3$. A point $\mathbf{x} = (x_0:\dots:x_4)\in S$ with $x_0,\ldots,x_4\in\mathcal{O}_K$ satisfying \eqref{corresponding_gcd_ideal} lies in $\mathcal{U}_3(\ringofintegers)$ if and only if  
		\begin{equation}\label{gcd_cond_i=3}
			x_0\ringofintegers + x_3\ringofintegers + x_4\ringofintegers = P_j
		\end{equation}
		for some $j=1,\dots,h_K$. The height is given by
		\begin{equation}\label{height_function_3}
			H_3(\mathbf{x}) = \frac{\max\{\norminf{x_0},\norminf{x_3}, \norminf{x_4}\}}{\idealnorm{x_0\ringofintegers + x_3\ringofintegers + x_4\ringofintegers}}.
		\end{equation}
		
		The method that we use to count integral points of bounded height with respect to the boundaries $Q$ and $L_2$, does not work with respect to the boundary $L_3$. Our attempts to use the same methods as for the other two cases fail in computing the error term of the first summation. 
		Hence, we will only parameterise the set of integral points for the third case, but not treat the resulting counting problem in the following section.
	\end{remark}

	\subsection{Integral Points on a Universal Torsor}
	This section is based on \cite{frei_pieropan_15}.
	The aim of this section is to apply \cite[Theorem 2.7]{frei_pieropan_15} to an $\ringofintegers$-model of a universal torsor of $\widetilde{S}_{\overline{K}}$ obtained by \cite[Construction 3.1]{frei_pieropan_15} to get a parameterisation of the set of $K$-integral points on the open subset $U_i$ via integral points on twisted torsors. 
	
	We describe the universal torsor $\overline{Y}\rightarrow\widetilde{S}_{\overline{K}}$ in the same turn as constructing an $\ringofintegers$-model of it, which is a universal torsor over a projective $\ringofintegers$-model of $\widetilde{S}_{\overline{K}}$. To this end, we consider the following monomials. For all $1\le i<j\le9$, let $A_{i,j}= \prod_{l\in\lbrace 1,\ldots,9\rbrace\setminus\{i,j\}}\eta_l$ and $A_{7,8,9}=\eta_1\cdots\eta_6$. Let $J$ be the ideal of $\mathrm{Cox}(\widetilde{S}_{\overline{K}})$ generated by the monomials
	\begin{gather}\label{polynomials f_i}
		A_{7,8,9},~A_{1,2},~A_{1,4},~A_{2,3},~A_{2,5},~A_{3,6},~A_{4,7},~A_{5,9},~\text{and}~A_{6,8}.
	\end{gather}
	These are obtained from the Dynkin-diagram in \Cref{fig:1-Dynkin-diagram} by considering the maximal subsets of vertices that are pairwise connected by at least one edge. We denote these polynomials by $f_1,\ldots,f_9$ and call $J$ the \emph{irrelevant ideal}.
	
	By \cite{derenthal14_lms}, we have $E_7\cap E_8\cap E_9\neq\emptyset$. Thus, the open subscheme $\overline{Y}$, which is defined to be the complement to $V(J)$ in $\Spec(\mathrm{Cox}(\widetilde{S}_{\overline{K}}))$, is a universal torsor of $\widetilde{S}_{\overline{K}}$ by \cite[Remark 6]{bourqui11}.
	Let
	\[R_{\ringofintegers}= \ringofintegers[\eta_1,\ldots,\eta_9]/(\eta_1\eta_4^2\eta_7 + \eta_3\eta_6^2\eta_8 + \eta_5\eta_9),\]
	$g = \eta_1\eta_4^2\eta_7 + \eta_3\eta_6^2\eta_8 + \eta_5\eta_9$, and let $\mathcal{Y}\rightarrow \widetilde{\mathcal{S}}$ be the $\ringofintegers$-model of the universal torsor $\overline{Y}\rightarrow \widetilde{S}_{\overline{K}}$ defined by $f_1,\ldots,f_9$ in \cite[Construction 3.1]{frei_pieropan_15}.
	We obtain an analogous result to \cite[Proposition 4.1]{frei_pieropan_15}:
	
	\begin{prop}\label{prop_Y_universal_torsor}
		\begin{enumerate}
			\item The scheme $\widetilde{\mathcal{S}}$ is smooth, projective, and with geometrically integral fibres over $\ringofintegers$.
			\item For every prime ideal $\pprime$ of $\ringofintegers$, the fibre $\widetilde{\mathcal{S}}_{k(\pprime)}$ is obtained from $\PP^2_{k(\pprime)}$ by a chain of 5 blowing-ups at $k(\pprime)$-points.
			\item The morphism $\mathcal{Y}\rightarrow \widetilde{\mathcal{S}}$ is a universal torsor under $\mathbb{G}^6_{m,\widetilde{\mathcal{S}}}$.
		\end{enumerate}
		\begin{proof}
			The proof is analogue to the proof of \cite[Proposition 4.1]{frei_pieropan_15} with some adaptions (see also \cite[Remark 4.4]{frei_pieropan_15}). The indices 1 and 6 have to be replaced by 5 and 2, respectively, i.e.\ for example the occurring sections $\eta_1$ and $\eta_6$ have to be replaced by $\eta_5$ and $\eta_2$, respectively.
			
			To prove that $\widetilde{\mathcal{S}}$ is smooth, as in \cite[Proposition 4.1]{frei_pieropan_15} we use \cite[Proposition 3.7]{frei_pieropan_15}. Therefore, we need to show that the Jacobian matrix $(\partial g/\partial \eta_i)_{1\le i\le9}$ has rank $1=9-2-6$ on $\mathcal{Y}(\overline{k(\pprime)})$, where $\overline{k(\pprime)}$ is an algebraic closure of the residue field $k(\pprime)$ of the prime ideal $\pprime$ in $\ringofintegers$. We have
			\[(\partial g/\partial\eta_i)_{1\le i\le 9} = 
			(\eta_4^2\eta_7,0,\eta_6^2\eta_8,2\eta_1\eta_4\eta_7,\eta_9,2\eta_3\eta_6\eta_8,\eta_1\eta_4^2,\eta_3\eta_6^2,\eta_5).\]
			Suppose $\eta_5=\eta_9=0$ on $\mathcal{Y}(\overline{k(\pprime)})$. Then, $f_i=0$ for all $1\le i\le9,i\neq8$, and $f_8 \neq 0$. This implies $\eta_j\neq0$ on $\mathcal{Y}(\overline{k(\pprime)})$ for all $j\neq 5,9$. Hence, the Jacobian matrix has rank 1 on $\mathcal{Y}(\overline{k(\pprime)})$. The rest of the proof remains unchanged.
		\end{proof}
	\end{prop}

	We have seen that the desingularisation $\widetilde{S}$ can be described as a certain sequence of five blowing-ups of $\PP^2_K$ in rational points \cite{derenthal14_compos}. The proof of \cite[Proposition 4.1]{frei_pieropan_15} with the slight modifications given in the proof of \Cref{prop_Y_universal_torsor} shows that the integral model $\widetilde{\mathcal{S}}$ can be defined by the same sequence of blowing-ups of $\PP^2_{\ringofintegers}$.
	\Cref{prop_Y_universal_torsor}(3) yields that $\mathcal{Y}$, which is defined before this proposition, is a $\mathbb{G}_{m,\ringofintegers}^6$-torsor over $\widetilde{\mathcal{S}}$ via a morphism $\rho:\mathcal{Y}\rightarrow\widetilde{\mathcal{S}}$.
	
	The action of $\mathbb{G}_{m,\ringofintegers}^6(\ringofintegers)= (\ringofintegers^\times)^6$ on $\mathcal{Y}(\ringofintegers)$ is given by \cite[(3.2)]{frei_pieropan_15} using the degrees from \eqref{degree_E_i}: an element \[((t_0,\ldots,t_5),(\eta_1,\ldots,\eta_9))\in(\ringofintegers^\times)^6\times\mathcal{Y}(\ringofintegers)\] maps to
	\begin{equation*}
		(t_1t_4^{-1}\eta_1, t_0t_1^{-1}t_2^{-1}t_3^{-1}\eta_2, t_2t_5^{-1}\eta_3,t_4\eta_4,t_3\eta_5, t_5\eta_6, t_0t_1^{-1}t_4^{-1}\eta_7,t_0t_2^{-1}t_5^{-1}\eta_8,t_0t_3^{-1}\eta_9).
	\end{equation*}
	
	Let $\rho\colon Y\rightarrow \widetilde{S}$ be the base change of the torsor morphism $\mathcal{Y}\rightarrow \widetilde{\mathcal{S}}$ from $\ringofintegers$ to $K$. Then, \cite[Remark 3.2]{frei_pieropan_15} yields that $\rho$ is a universal torsor of $\widetilde{S}$.
	As in \cite[Section 4]{frei_pieropan_15}, we obtain a morphism $\Psi\colon Y\rightarrow S$ which is the composition of $\rho$ and $\pi$ (see also \cite[Remark 4.4]{frei_pieropan_15} where Frei and Pieropan state that their constructions also work for the other examples in \cite{derenthal14_lms}).
	The map $\Psi$ is given by sending $(\eta_1,\ldots,\eta_9)\in Y(K)$ to the point
	\[(\eta_1^2\eta_2^2\eta_3\eta_4^2\eta_5\eta_7:
	\eta_1\eta_2^2\eta_3^2\eta_5\eta_6^2\eta_8:
	\eta_1^2\eta_2^3\eta_3^2\eta_4\eta_5^2\eta_6:
	\eta_1\eta_2\eta_3\eta_4\eta_6\eta_7\eta_8:
	\eta_7\eta_8\eta_9)\]
	in $S(K)\subseteq\PP^4(K)$,
	where the sections \begin{equation}\label{sections_L0}
		L_0 = \lbrace \eta_1^2\eta_2^2\eta_3\eta_4^2\eta_5\eta_7,
		\eta_1\eta_2^2\eta_3^2\eta_5\eta_6^2\eta_8,
		\eta_1^2\eta_2^3\eta_3^2\eta_4\eta_5^2\eta_6,
		\eta_1\eta_2\eta_3\eta_4\eta_6\eta_7\eta_8,\eta_7\eta_8\eta_9\rbrace
	\end{equation}
	have anticanonical degree.
	
	We can use this to give an explicit parameterisation of $U_i(K)$, for $i=1,2,3$, by integral points on twists of $\mathcal{Y}$. But first, we describe the preimage of $V$ inside the torsor.
	
	Recall that $V$ is defined as the complement of the lines in $S$ and that 
	\[S\setminus V = S\setminus\{x_0x_3=0\}.\]
	An easy computation shows that $\Psi^{-1}(S\setminus V) = \lbrace \eta_1\cdots\eta_8=0\rbrace$, and 
	\begin{equation}\label{preimage_Psi(V(K))}
		(\Psi^{-1}(V))(K) = Y(K)\cap \left((K^\times)^8\times K\right).
	\end{equation}
	
	Analogously to \cite{derenthal14_compos}, for any given 6-tuple $\bm{C}=(C_0,\ldots,C_5)$ of nonzero fractional ideals of $\mathcal{O}_K$ (for example $\bm{C}\in\mathcal{C}^6$) we define 
	\begin{equation*}
		u_\mathbf{C}= 
		\idealnorm{C_0^3C_1^{-1}\cdots C_5^{-1}},
	\end{equation*} 
	and $\mathcal{O}_j = \bm{C}^{\deg\eta_j}$ for $j=1,\dots,9$, that means,
	\begin{align}\label{def_O_i}
		\begin{array}{lll}
			\mathcal{O}_1= C_1C_4^{-1},
			& \mathcal{O}_2= C_0C_1^{-1}C_2^{-1}C_3^{-1},
			&\mathcal{O}_3= C_2C_5^{-1},\\
			\mathcal{O}_4= C_4,
			& \mathcal{O}_5= C_3,
			&\mathcal{O}_6= C_5,\\
			\mathcal{O}_7= C_0C_1^{-1}C_4^{-1},
			&\mathcal{O}_8= C_0C_2^{-1}C_5^{-1}, \ \text{and}
			&\mathcal{O}_9= C_0C_3^{-1}.
		\end{array}
	\end{align}
	Let
	\begin{align*}
		\mathcal{O}_{j*}= \begin{cases}
			\mathcal{O}_j^{\neq 0} & \text{if }j\in\{1,\ldots,8\},\\
			\mathcal{O}_j & \text{if } j=9.
		\end{cases}
	\end{align*}
	For $\eta_j\in\mathcal{O}_j$, we define 
	\begin{equation}\label{def_I_j}
		I_j =  I_j(\eta_j) =\eta_j\mathcal{O}_j^{-1}\subseteq\ringofintegers.
	\end{equation}
	For simplicity, let $\bm{\eta} = (\eta_1,\ldots,\eta_9)$. 
	Let ${}_{\bm{C}}\rho\colon {}_{\bm{C}}\mathcal{Y}\rightarrow \widetilde{\mathcal{S}}$ be the twist of $\mathcal{Y}$ as constructed in \cite[Definition 2.6]{frei_pieropan_15}. By \cite[Theorem 2.5(i)]{frei_pieropan_15} these twists ${}_{\bm{C}}\mathcal{Y}$ are different integral models of the $K$-variety $Y$.
	
	As in \cite[Section 4]{frei_pieropan_15}, we give an explicit parameterisation of rational points by lattice points.
	
	\begin{lemma}\label{lemma_parameterisation_of_rational_points_via_universal_torsor}
		For any given $\bm{C}\in\mathcal{C}^6$, the map ${}_{\bm{C}}\rho$ induces a $\omega_K^6$-to-$1$-correspondence
		\[\bigsqcup_{\bm{C}\in\mathcal{C}^6} {}_{\bm{C}}\mathcal{Y}(\ringofintegers)\cap (\Psi^{-1}(V))(K) \rightarrow \widetilde{V}(K).\]
		We further have that ${}_{\bm{C}}\mathcal{Y}(\ringofintegers)\cap(\Psi^{-1}(V))(K)$ is the set of all $\bm{\eta}\in\mathcal{O}_{1*}\times\cdots\times\mathcal{O}_{9*}$
		such that
		\begin{align}
			&\eta_1\eta_4^2\eta_7 + \eta_3\eta_6^2\eta_8 + \eta_5\eta_9 = 0,\text{ and} \label{torsor-equation_new}\\
			&I_i + I_j = \ringofintegers \text{ if }E_i\text{ and }E_j\text{ do not share an edge in \Cref{fig:1-Dynkin-diagram}.} \label{gcd-condition_eta}
		\end{align}
		\begin{proof}
			This proof is based on \cite[Lemma 4.3]{frei_pieropan_15}.
			Recall that $\mapmindesingmodel\colon\widetilde{\mathcal{S}}\rightarrow \mathcal{S}$ denotes the minimal desingularisation, which is a model of the desingularisation $\pi\colon\widetilde{S}_{\overline{K}}\rightarrow S_{\overline{K}}$ induced by the anticanonical sections \eqref{sections_L0} of $\widetilde{S}$. This morphism induces an isomorphism $\mapmindesingmodel^{-1}(V)\rightarrow V$.
			By \cite[Theorem 2.7(ii)]{frei_pieropan_15} the set of rational points on the open variety $\pi^{-1}(V) = \widetilde{V}$ can be written as a disjoint union
			\[(\mapmindesingmodel^{-1}(V))(K) = \bigsqcup_{\bm{C}\in\mathcal{C}^6}{}_{\bm{C}}\rho({}_{\bm{C}}\mathcal{Y}(\ringofintegers)\cap(\Psi^{-1}(V))(K)).\]
			
			Now, let $\bm{C}\in\mathcal{C}^6$. \cite[Theorem 2.7(iii)]{frei_pieropan_15} and \eqref{preimage_Psi(V(K))} yield that ${}_{\bm{C}}\mathcal{Y}(\ringofintegers)\cap(\Psi^{-1}(V))(K)$ is the set of all
			\[\bm{\eta}\in \mathcal{O}_{1^*}\times\dots\times\mathcal{O}_{9^*}\]
			satisfying \eqref{torsor-equation_new} and 
			\begin{equation}\label{equivalent_coprime_condition}
				\sum_{i=1}^{9}f_i(\mathbf{I}) = \ringofintegers.
			\end{equation}
			Here, for every $\bm{\eta}$ we set $\mathbf{I} = \mathbf{I}(\bm{\eta}) = (I_1,\dots,I_9)$. The $f_i$ are the polynomials defined in \eqref{polynomials f_i}, that is, for instance $f_1(\mathbf{I}) = I_1\cdots I_6$.
			Analogously to \cite[Proof of Lemma 4.3]{frei_pieropan_15} one shows that \eqref{equivalent_coprime_condition} is equivalent to the coprimality conditions \eqref{gcd-condition_eta}.
		\end{proof}
	\end{lemma}
	
	With the use of the previous lemma, we can finally give an explicit parameterisation of integral points by lattice points: in addition to the ``coprimality conditions'' $I_j+I_k = \ringofintegers$, we get conditions $I_j=\ringofintegers$ for $E_j\subset\vert D_i\vert$.
	
	\begin{prop}\label{lemma_parameterisation_of_integer_points_via_universal_torsor}
		Let $i\in\lbrace1,2,3\rbrace$. For a given $\bm{C}\in\mathcal{C}^6$ we set ${}_{\bm{C}}\mathcal{Y}_{i}= {}_{\bm{C}}\rho^{-1}(\widetilde{\mathcal{U}}_i)\subset {}_{\bm{C}}\mathcal{Y}$.
		Then, ${}_{\bm{C}}\rho\colon {}_{\bm{C}}\mathcal{Y}_{i}\rightarrow \widetilde{\mathcal{U}}_i$ is a $\mathbb{G}_{m,\ringofintegers}^6$-torsor, which induces a $\omega_K^6$-to-$1$-correspondence
		\begin{equation*}
			\bigsqcup_{\substack{\bm{C}\in\mathcal{C}^6}}{}_{\bm{C}}\mathcal{Y}_{i}(\ringofintegers)\cap (\Psi^{-1}(V))(K) \rightarrow \widetilde{\mathcal{U}}_i(\ringofintegers)\cap \widetilde{V}(K).
		\end{equation*}
		Explicitly, we have
		\begin{equation*}
			{}_{\bm{C}}\mathcal{Y}_{i}(\ringofintegers)\cap (\Psi^{-1}(V))(K)
			=\lbrace (\eta_1,\dots,\eta_9)\in \mathcal{O}_{1*}\times\dots\times \mathcal{O}_{9*} \mid \eqref{torsor-equation_new},\eqref{gcd-condition_eta},
			\eqref{I_j=O_K}\rbrace,
		\end{equation*}
		where \eqref{torsor-equation_new} and \eqref{gcd-condition_eta} are given in the previous lemma, and
		\begin{equation}
			I_j=\ringofintegers \text{ if }E_j\subset\lvert D_i\rvert\label{I_j=O_K}.
		\end{equation}
	\end{prop}
	\begin{proof}
		The restriction of $\rho$ to the open subscheme $\widetilde{\mathcal{U}}_i$ is a $\mathbb{G}_{m,\ringofintegers}^6$-torsor, since $\rho$ itself is a\linebreak $\mathbb{G}_{m,\ringofintegers}^6$-torsor.
		\Cref{lemma_parameterisation_of_rational_points_via_universal_torsor}, the fact that $\widetilde{V} = \pi^{-1}(V)$, and $\Psi = \pi\circ\rho$ give us the stated correspondence.

		For simplicity, we first study what happens if we remove an irreducible divisor $E_j$ from $\widetilde{S}$. Define \[Y_j = Y\setminus \rho^{-1}(E_j) = Y\setminus V(\eta_j)\] 
		for $j=1,\ldots,8$.
		Note that the second equality holds due to \cite[Proposition 1.6.2.1]{ADHL15} (or \cite[Proposition 1.5.3.6]{ADHL15}). For each $\bm{C}\in\mathcal{C}^6$ we shall also need the twists 
		\[{}_{\bm{C}}\widetilde{\mathcal{Y}}_j = {}_{\bm{C}}\mathcal{Y}\setminus V({}_{\bm{C}}(\eta_j))\] in the sense of \cite[Definition 2.4]{frei_pieropan_15}. 
		Now, \begin{equation}\label{eq:different_characterisation_twist_Y_i}
			{}_{\bm{C}}\widetilde{\mathcal{Y}}_j = {}_{\bm{C}}\mathcal{Y}\setminus \overline{{}_{\bm{C}}\rho^{-1}(E_j)}.\end{equation} 
		Indeed, $V({}_{\bm{C}}(\eta_j))\subset{}_{\bm{C}}\mathcal{Y}$ is a closed set containing $V(\eta_j)\subset Y = {}_{\bm{C}}\mathcal{Y}_K$ (see \cite[Theorem 2.5(i)]{frei_pieropan_15}) and $\overline{{}_{\bm{C}}\rho^{-1}(E_j)}$ is the smallest closed subset of ${}_{\bm{C}}\mathcal{Y}$ that contains $\rho^{-1}(E_j) = V(\eta_j)$. Thus, \[\overline{{}_{\bm{C}}\rho^{-1}(E_j)}\subset V({}_{\bm{C}}(\eta_j)).\] 
		Consider the morphism \[\varphi\colon \overline{{}_{\bm{C}}\rho^{-1}(E_j)}\hookrightarrow V({}_{\bm{C}}(\eta_j))\]
		in ${}_{\bm{C}}\mathcal{Y}$.
		For each class $[P_i]$ of the class group $\mathrm{Cl}_K$ of $K$, choose two prime ideals $\pprime_j\neq \mathfrak{q}_j$ with $[\pprime_j] = [\mathfrak{q}_j] = [P_j]$, $j=1,\dots, h_K$.
		Indeed, applying Chebotarev's density theorem to the Hilbert class field of $K$ yields infinitely many prime ideals in every class $[P_i]\in\mathrm{Cl}_K$.
		Then, \begin{align*}
			\mathcal{W}_1 &= \Spec(\ringofintegers)\setminus \{\pprime_1,\dots,\pprime_{h_K}\}\quad\text{and}\\
			\mathcal{W}_1 &= \Spec(\ringofintegers)\setminus \{\mathfrak{q}_1,\dots,\mathfrak{q}_{h_K}\}
		\end{align*}
		cover $\Spec(\ringofintegers)$ by construction, and the sequences
		\begin{align*}
			\ZZ\pprime_1\oplus\dots\oplus\ZZ\pprime_{h_K}&\to \Pic\Spec(\ringofintegers) \rightarrow\Pic\mathcal{W}_1\rightarrow 0\\
			\ZZ\mathfrak{q}_1\oplus\dots\oplus\ZZ\mathfrak{q}_{h_K}&\to \Pic\Spec(\ringofintegers) \rightarrow\Pic\mathcal{W}_2\rightarrow 0
		\end{align*}
		are exact. As the morphisms on the left are surjective, 
		both Picard groups $\Pic\mathcal{W}_i$ vanish.
		Now, \cite[Theorem 2.5.(ii)]{frei_pieropan_15} applied to the affine open covering $\mathcal{W}_1\cup\mathcal{W}_2$ of $\Spec(\ringofintegers)$ yields that $\varphi$ is an isomorphism on this open covering and hence on $\Spec(\ringofintegers)$. The identity \eqref{eq:different_characterisation_twist_Y_i} follows.

		Let $(\eta_j)_m$ denote the degree-$m$-part of the ideal $(\eta_j)$. Due to \cite[Theorem 2.5(iii)]{frei_pieropan_15} and \Cref{lemma_parameterisation_of_rational_points_via_universal_torsor}, we obtain
		\[{}_{\bm{C}}\widetilde{\mathcal{Y}}_j(\ringofintegers) = \{\bm{\eta}\in {}_{\bm{C}}\mathcal{Y}(\ringofintegers)\mid \eqref{eqn_defining_condition_integral_points} \},\]
		where
		\begin{equation}\label{eqn_defining_condition_integral_points}
			\sum_{m\in\ZZ^6}\sum_{f\in(\eta_j)_m} f(\bm{\eta})\bm{C}^{-m} = \ringofintegers.
		\end{equation}
		It suffices to consider the generator $\eta_j$ of the ideal $(\eta_j)$ on the left hand side of \eqref{eqn_defining_condition_integral_points}. Thus, this condition is equivalent to $\eta_j\bm{C}^{-\deg(\eta_j)} = \ringofintegers$, that is, to $I_j = \ringofintegers$ by definition of the $\mathcal{O}_j$ and $I_j$, see \eqref{def_O_i} and \eqref{def_I_j}.
		
		Now, 
		\[{}_{\bm{C}} \mathcal{Y}_i = \bigcap_{\substack{j\\E_j\subset\vert D_i\vert}} {}_{\bm{C}} \widetilde{\mathcal{Y}}_j.\]
		Therefore, we obtain
		\[ {}_{\bm{C}}\mathcal{Y}_i(\ringofintegers) = \bigcap_{\substack{j\\E_j\subset\vert D_i\vert}} {}_{\bm{C}}\widetilde{\mathcal{Y}}_j(\ringofintegers) = \{\bm{\eta}\in{}_{\bm{C}}\mathcal{Y}(\ringofintegers)\mid I_j = \ringofintegers~\text{for all}~ j\text{ with }E_j\subset\vert D_i\vert\},\]
		which together with \Cref{lemma_parameterisation_of_rational_points_via_universal_torsor} proves the lemma.
	\end{proof}
	
	\begin{remark}
		In comparison to the abstract definition of integral points on the universal torsor in \cite[Definition 2.4]{frei_pieropan_15}, there is also a more elementary and intuitive characterisation of integral points. 
		Let $\bm{\eta}$ be an integral point in ${}_{\bm{C}}\mathcal{Y}(\ringofintegers)\cap(\Psi^{-1}(V))(K)$, that is,  $(\eta_1,\dots,\eta_9)\in\mathcal{O}_{1*}\times\cdots\times\mathcal{O}_{9*}$ such that \eqref{torsor-equation_new} and \eqref{gcd-condition_eta} hold, due to \Cref{lemma_parameterisation_of_rational_points_via_universal_torsor}. For $i=1,2,3$, computations show that the point $(\eta_1,\dots,\eta_9)$ lies in ${}_{\bm{C}}\mathcal{Y}_i(\ringofintegers)$ if and only if for every prime ideal $\pprime\subseteq\ringofintegers$ there exists an element $\bm{t}\in\mathbb{G}_m^6(K)$ such that $\bm{t}\bm{\eta} = (\eta_1',\dots,\eta_9')$ satisfies 
		\begin{align}\label{integrality_conditions_mod_p}
			\begin{split}
				& v_\pprime(\eta_j')\ge 0 \text{ for all } j=1,\ldots,9, \\
				&\eta_j'\not\equiv 0 \;\bmod\;\mathfrak{p}\quad
				\begin{cases}
					\text{ for }j=1,2,3 &\text{if }i=1,\\
					\text{ for } j=1,2,3,4 &\text{if } i=2,\\
					\text{ for } j=7 &\text{if } i=3,
				\end{cases}\quad\text{and}\\
				& (\eta_j'\;\bmod\;\mathfrak{p},\eta_k'\;\bmod\;\mathfrak{p})\not\equiv(0,0) \text{ for all } E_j,E_k \text{ that do not share an edge in \Cref{fig:1-Dynkin-diagram}.}
			\end{split}
		\end{align}
		The basic approach to prove this is to translate \eqref{integrality_conditions_mod_p} into a system of linear equations and inequalities by using the $\pprime$-adic valuation $v_\pprime$ of $\eta_j$ and $t_j$, which then can be solved by using linear algebra.
	\end{remark}

	To get a nicer representation of the set of integral points, which makes the counting easier, we choose a different (but isomorphic) twist $_{\bm{C}'}\mathcal{Y}$ of the universal torsor $\mathcal{Y}$. To this end, we define
	\begin{align*}
		\mathcal{C}_1 &= \{(C_0,\dots,C_5)\mid C_3,C_4,C_5\in\mathcal{C}, C_0 = C_3C_4C_5, C_1 = C_4, C_2 = C_5\},\\
		\mathcal{C}_2 &= \{(C_0,\dots,C_5)\mid C_3,C_5\in\mathcal{C}, C_0 = C_3C_5, C_1 = C_4 = \ringofintegers, C_2 = C_5\},\quad\text{and}\\
		\mathcal{C}_3 &= \{(C_0,\dots,C_5)\mid C_1,\dots,C_5\in\mathcal{C}, C_0 = C_1C_4\}.
	\end{align*}
	We note that $_{\bm{C}}\mathcal{Y}_i$, defined in \Cref{lemma_parameterisation_of_integer_points_via_universal_torsor}, can be defined more generally for any $6$-tuple $\bm{C}$ of nonzero fractional ideals of $\ringofintegers$ (see also \cite[Definition 2.2]{frei_pieropan_15}), that means, not all coordinates of $\bm{C}$ have to be elements of $\mathcal{C}$. 
	We obtain the following nicer representation.
	
	\begin{prop}\label{lemma_nicer_representation}
		Let $i\in\lbrace1,\dots,3\rbrace$. Let $\bm{C}'\in\mathcal{C}_i$. 
		Then, ${}_{\bm{C}'}\rho\colon {}_{\bm{C}'}\mathcal{Y}_{i}\rightarrow \widetilde{\mathcal{U}}_i$ is a $\mathbb{G}_{m,\ringofintegers}^6$-torsor, which induces a $\omega_K^6$-to-$1$-correspondence
		\begin{equation*}
			\bigsqcup_{\substack{\bm{C}'\in\mathcal{C}_i}}{}_{\bm{C}'}\mathcal{Y}_{i}(\ringofintegers)\cap (\Psi^{-1}(V))(K) \rightarrow \widetilde{\mathcal{U}}_i(\ringofintegers)\cap \widetilde{V}(K).
		\end{equation*}
		Explicitly, we have
		\begin{equation*}
			{}_{\bm{C}'}\mathcal{Y}_{i}(\ringofintegers)\cap (\Psi^{-1}(V))(K)
			=\lbrace (\eta_1,\dots,\eta_9)\in \mathcal{O}_{1*}\times\dots\times \mathcal{O}_{9*} \mid \eqref{torsor-equation_new},\eqref{gcd-condition_eta},
			\eqref{uni-condition}\rbrace,
		\end{equation*}
		where \eqref{torsor-equation_new} and \eqref{gcd-condition_eta} are given in \Cref{lemma_parameterisation_of_rational_points_via_universal_torsor}, and
		\begin{equation}
			\eta_j\in\ringofintegers^\times \text{ if }E_j\subset\lvert D_i\rvert. \label{uni-condition}
		\end{equation}
		\begin{proof}
			The fact that ${}_{\bm{C}'}\rho$ is a $\mathbb{G}_{m,\ringofintegers}^6$-torsor is proven analogously to the previous proposition. To prove the stated correspondence and representation of integral points, let $\bm{C}\in\mathcal{C}^6$. 
			
			We start with the case $i=1$. The previous proposition gives us that an integral point $\bm{\eta}\in {}_{\bm{C}}\mathcal{Y}(\ringofintegers)\cap(\Psi^{-1}(V))(K)$ lies in ${}_{\bm{C}}\mathcal{Y}_{1}(\ringofintegers)\cap (\Psi^{-1}(V))(K)$ if and only $I_1 = I_2 = I_3 = \ringofintegers$.
			By \eqref{def_I_j}, we have $\eta_1\mathcal{O}_1^{-1} = \ringofintegers$, and equivalently $(\eta_1) = \mathcal{O}_1$.
			Since $(\eta_1)$ is a principal ideal, by \eqref{def_O_i} this is only possible if $[C_1C_4^{-1}] = [\ringofintegers]$. Analogously, we obtain $(\eta_2) = \mathcal{O}_2$, which implies $[C_0C_1^{-1}C_2^{-1}C_3^{-1}] = [\ringofintegers]$, as well as $(\eta_3) = \mathcal{O}_3$, which implies $[C_2C_5^{-1}]=[\ringofintegers]$.
			Since we have chosen only one ideal in each ideal class, we obtain $C_1=C_4$ and $C_2=C_5$. Hence, $\eta_1$ and $\eta_3$ have to be elements in $\ringofintegers^\times$.
			Further, $C_0$ is uniquely determined by $[C_0] = [C_3C_4C_5]$. 
			Due to \cite[Proposition 2.5(iv)]{frei_pieropan_15} we can consider the to $_{\bm{C}}\mathcal{Y}$ isomorphic twist $_{\bm{C}'}\mathcal{Y}$ with $\bm{C}' = (C_3C_4C_5,C_4,C_5,C_3,C_4,C_5)$, that means we can replace $C_0$ with $C_3C_4C_5$ and $\bm{\eta}$ maps to itself. In this twist, we obtain $(\eta_2) = \ringofintegers$, and therefore $\eta_2\in\ringofintegers^\times$. By construction, the image of this twist in $\widetilde{\mathcal{U}}_1(\ringofintegers)\cap\widetilde{V}(K)$ remains unchanged. Therefore, together with \Cref{lemma_parameterisation_of_integer_points_via_universal_torsor} the stated correspondence and representation of integral points follows. We only have to take the disjoint union over $\bm{C}'\in\mathcal{C}_1\cong \mathcal{C}^{3}$, as we have shown that the remaining twists do not contain any integral points.
			
			In the case $i=2$, additionally to the case $i=1$, we obtain the condition $[C_4]=[\ringofintegers]$ by $I_4 = \ringofintegers$, which implies $C_1=C_4=\ringofintegers$. Again, we consider an isomorphic twist $_{\bm{C}'}\mathcal{Y}$ of $_{\bm{C}}\mathcal{Y}$ by choosing $\bm{C}' = (C_3C_5,\ringofintegers,C_5,C_3,\ringofintegers,C_5)$. 
			
			For an integral point $\bm{\eta}$ in ${}_{\bm{C}}\mathcal{Y}_{3}(\ringofintegers)\cap (\Psi^{-1}(V))(K)$, the previous proposition gives us $I_7 = \ringofintegers$. Thus, $[C_0C_1^{-1}C_4^{-1}]=[\ringofintegers]$, or equivalently $[C_0] = [C_1C_4]$.
			By choosing the isomorphic twist $_{\bm{C}'}\mathcal{Y}$ of $_{\bm{C}}\mathcal{Y}$ with $\bm{C}' = (C_1C_4, C_1,C_2,C_3,C_4,C_5)$, we get rid of $C_0$ and the lemma follows.
		\end{proof}
	\end{prop}

	\begin{example}\label{example:integral_no_subset_of_rational}
		We want to emphasise that the parameterisation of integral points on the universal torsor $\mathcal{Y}$ in \Cref{lemma_parameterisation_of_integer_points_via_universal_torsor} contains points $(\eta_1,\dots,\eta_9)$ with $\eta_j\not\in\ringofintegers^\times$ for $E_j\subset\vert D_i\vert$, so that the nicer parameterisation in \Cref{lemma_nicer_representation} really changes the set of integral points. To this end, we consider an explicit example.
		
		Let $K = \QQ(\alpha)$ with $\alpha = \sqrt{-5}$ be a number field of class number $h_K = 2$ with ring of integers $\ringofintegers = \ZZ[\alpha]$. Take $\mathfrak{p} = (2,1+\alpha)$ and let $\mathcal{C} = \{\ringofintegers, \mathfrak{p}\}$. 
		For $\bm{C} = (\mathfrak{p},\dots,\mathfrak{p})$ and $i=1$, consider the point
		\begin{equation}\label{example_point_p}
			P = \left(1:\tfrac{1}{2}:1:1+\alpha:2:1-\alpha:\tfrac{1}{2}+\tfrac{1}{2}\alpha:\tfrac{1}{2}-\tfrac{1}{2}\alpha:7\right).
		\end{equation}
		Let $P'$ denote its representative in $K^9$ defined by the same coordinates as given above. We set $\pprime_3 = (3,1+\alpha)$, $\overline{\pprime}_3 = (3,1-\alpha)$, $\pprime_7 = (7,3+\alpha)$ and $\overline{\pprime}_7 = (7,3-\alpha)$.
		One easily checks that the entries of $P'$ have the following factorisations into prime ideals
		\begin{align*}
			\left(\tfrac{1}{2}\right) = \mathfrak{p}^{-2},\quad
			\left(1+\alpha\right) = \mathfrak{p}\pprime_3,\quad
			(2) = \mathfrak{p}^2,\quad
			\left(1-\alpha\right) = \mathfrak{p}\overline{\pprime}_3,\\
			\left(\tfrac{1}{2}+\tfrac{1}{2}\alpha\right) = \mathfrak{p}^{-1}\pprime_3,\quad
			\left(\tfrac{1}{2}-\tfrac{1}{2}\alpha\right) = \mathfrak{p}^{-1}\overline{\pprime
			}_3,\quad\text{and}\quad
			(7) = \pprime_7\overline{\pprime}_7.
		\end{align*}
		This shows that
		\begin{align*}
			\begin{array}{lll}
				I_1 = \ringofintegers,& I_2 = \ringofintegers, &I_3 = \ringofintegers,\\
				I_4 = \pprime_3,& I_5 = \mathfrak{p},& I_6 = \overline{\pprime}_3,\\ 
				I_7 = \pprime_7,& I_8 = \overline{\pprime}_3,\text{ and} & I_9 = \pprime_7\overline{\pprime}_7.
			\end{array}
		\end{align*}
		We deduce that \eqref{gcd-condition_eta} is satisfied. Moreover, the torsor equation \eqref{torsor-equation_new} holds, since
		\[2\cdot 7 + \tfrac{1}{2}(1+\alpha)^3 + \tfrac{1}{2}(1-\alpha)^3 = 0.\] 
		\Cref{lemma_parameterisation_of_integer_points_via_universal_torsor} shows that $P'$ is a point in ${}_{\bm{C}}\mathcal{Y}_{1}(\ringofintegers)\cap (\Psi^{-1}(V))(K)$ and therefore corresponds to an integral point on $S$.
		
		Clearly, $\eta_2\not\in\ringofintegers^\times$. Thus, \Cref{lemma_nicer_representation} yields that $P'$ is no element of ${}_{\bm{C}'}\mathcal{Y}_{1}(\ringofintegers)\cap (\Psi^{-1}(V))(K)$ (note that $\bm{C}' = (\pprime^3,\pprime,\dots,\pprime)$). Hence, $P'$ does not lie in the parameterisation of integral points from \Cref{lemma_nicer_representation}.
		Instead, \Cref{lemma_nicer_representation} represents $P$ by 
		\[P'' = (1,1,1,1+\alpha,2,1-\alpha, 1+\alpha, 1-\alpha, 14),\]
		which is obtained by acting on $P'$ with $\bm{t} = (2,1,1,1,1,1)$.
		
		We further note that $P'$ satisfies \eqref{integrality_conditions_mod_p} with $\bm{t} = 1$ for all prime ideals except $\pprime$. By taking $t_1=\dots=t_5=\tfrac{1}{2}+\tfrac{1}{2}\alpha$ and $t_0 = -\tfrac{7}{2}-\tfrac{1}{2}\alpha$ for $\pprime$ (we note that $t_i\ringofintegers = \mathfrak{p}^{-1}\pprime_3$ for $i=1,\dots,5$, and $t_0\ringofintegers = \mathfrak{p}^{-1}\pprime_3^3$), we have 
		\[\bm{t}P' = (1,1,1,-2+\alpha,1+\alpha,3,-2+\alpha,3,14).\]
		By considering the corresponding factorisations of the entries into prime ideals
		\begin{equation*}
			(-2+\alpha) = \pprime_3^2,\quad
			(1+\alpha) = \mathfrak{p}\pprime_3,\quad
			(3) = \pprime_3\overline{\pprime}_3,\quad\text{and}\quad
			(14) = \mathfrak{p}^2\pprime_7\overline{\pprime}_7
		\end{equation*}
		we conclude that $\bm{t}P'$ satisfies \eqref{integrality_conditions_mod_p} for $\mathfrak{p}$ (but not for $\pprime_3$).
	\end{example}

	\subsection{Log-anticanonical bundles and associated height functions}
	Now, we study the log-anticanonical bundles and the associated height functions. Recall that, due to symmetry reasons, the cases concerning $D_4$ and $D_5$ can be reduced to $D_2$ and $D_3$, respectively.
	
	\begin{lemma}\label{gcd_sets_M}
		The only nonzero reduced effective divisors $D\subset\widetilde{S}$ such that $\omega_{\widetilde{S}}(D)^\vee$ is big and nef are $D_i$ for $i\in\{1,\ldots,5\}$. 
		
		Consider the sets
		\begin{align*}
			M_1 &= \lbrace \eta_1\eta_2\eta_4^2\eta_5\eta_7, \eta_2\eta_3\eta_5\eta_6^2\eta_8,\eta_1\eta_2^2\eta_3\eta_4\eta_5^2\eta_6,\eta_4\eta_6\eta_7\eta_8 \rbrace,\\
			M_2 &=\lbrace \eta_1\eta_2\eta_4\eta_5\eta_7,\eta_1\eta_2^2\eta_3\eta_5^2\eta_6,\eta_6\eta_7\eta_8 \rbrace, \text{ and }\\
			M_3 &= \lbrace \eta_1^2\eta_2^2\eta_3\eta_4^2\eta_5,\eta_1\eta_2\eta_3\eta_4\eta_6\eta_8,\eta_8\eta_9 \rbrace
		\end{align*}
		of monomials in the Cox ring $R$ of degree $\omega_{\widetilde{S}}(D_i)^\vee$ for $i=1,2,3$, respectively. For any $6$-tuple $\bm{C}$ of nonzero fractional ideals of $\mathcal{O}_K$, and for $\bm{\eta}\in \mathcal{O}_{1}\times\cdots\times\mathcal{O}_9$ satisfying \eqref{gcd-condition_eta}, the greatest common divisor of the set $M_i$ is the ideal 
		\begin{align*}
			\begin{cases}
				C_0^2C_1^{-1}C_2^{-1} & \text{if }i=1,\\
				C_0^2C_1^{-1}C_2^{-1}C_4^{-1} & \text{if } i=2,\\
				C_0^2C_2^{-1}C_3^{-1}C_5^{-1} & \text{if } i=3.
			\end{cases}
		\end{align*}
	\end{lemma}
	\begin{proof}
		The first statement is a special case of \cite[Theorem 10]{derenthal_wilsch_21}.
		For the first set, a simple computation shows
		\begin{align*}
			\eta_1\eta_2\eta_4^2\eta_5\eta_7\ringofintegers+
			\eta_2\eta_3\eta_5\eta_6^2\eta_8\ringofintegers+
			\eta_1\eta_2^2\eta_3\eta_4\eta_5^2\eta_6\ringofintegers+
			\eta_4\eta_6\eta_7\eta_8\ringofintegers\\
			= C_0^2C_1^{-1}C_2^{-1}(I_1I_2I_4^2I_5I_7+I_2I_3I_5I_6^2I_8+I_1I_2^2I_3I_4I_5^2I_6+I_4I_6I_7I_8).
		\end{align*}
		We want to show 
		\begin{equation}\label{gcd_cond_M1}
			I_1I_2I_4^2I_5I_7+I_2I_3I_5I_6^2I_8+I_1I_2^2I_3I_4I_5^2I_6+I_4I_6I_7I_8 = \ringofintegers.
		\end{equation}
		Assume that $\mathfrak{p}\divides I_4I_6I_7I_8$ for a prime ideal $\pprime\subset\ringofintegers$. Then, we distinguish four cases. If $\mathfrak{p}\divides I_4$, then $\mathfrak{p}\nmid I_2I_3I_5I_6I_8$, since the corresponding divisors $E_2,E_3,E_5,E_6,E_8$ do not share an edge with $E_4$ in \Cref{fig:1-Dynkin-diagram}. Thus, the second addend is not divisible by $\mathfrak{p}$. If $\mathfrak{p}\divides I_6$, then $\mathfrak{p}\nmid I_1I_2I_4I_5I_7$. Hence, the first addend is not divisible by $\mathfrak{p}$. If $\mathfrak{p}\divides I_7$, then $\mathfrak{p}\nmid I_1I_2I_3I_5I_6$. And it can only divide either $I_4$ or $I_8$, because the corresponding divisors $E_4$ and $E_8$ do not share an edge in \Cref{fig:1-Dynkin-diagram}. Therefore, either the second or the third addend is not divisible by $\mathfrak{p}$. Lastly, assume $\mathfrak{p}\divides I_8$. Then, $\mathfrak{p}\nmid I_1\cdots I_5$ and it divides either $I_6$ or $I_7$. Thus, either the first or the third addend is not divisible by $\mathfrak{p}$. This proves \eqref{gcd_cond_M1}, and hence the statement for the first set.
		
		A very similar argument as above shows that \eqref{gcd_cond_M1} implies
		\begin{equation*}
			I_1I_2I_4I_5I_7+I_1I_2^2I_3I_5^2I_6+I_6I_7I_8 = \ringofintegers.
		\end{equation*}
		Then, we obtain
		\begin{equation*}
			\eta_1\eta_2\eta_4\eta_5\eta_7\ringofintegers+
			\eta_1\eta_2^2\eta_3\eta_5^2\eta_6\ringofintegers+
			\eta_6\eta_7\eta_8\ringofintegers
			= C_0^2C_1^{-1}C_2^{-1}C_4^{-1}.
		\end{equation*}
		
		For the third set, assume $\mathfrak{p}\divides I_8$ for a prime ideal $\mathfrak{p}\subset\ringofintegers$. Then, $\mathfrak{p}\nmid I_1\cdots I_5$ and $I_1^2I_2^2I_3I_4^2I_5$ is not divisible by $\mathfrak{p}$. If $\mathfrak{p}\divides I_9$, we have $\mathfrak{p}\nmid I_1\cdots I_4I_6$. 
		And $\mathfrak{p}$ divides either $I_5$ or $I_8$, since the corresponding divisors $E_5$ and $E_8$ do not share an edge in \Cref{fig:1-Dynkin-diagram}. Thus, we get
		\[I_1^2I_2^2I_3I_4^2I_5+I_1I_2I_3I_4I_6I_8+I_8I_9=\ringofintegers,\]
		and therefore,
		\begin{equation*}
			\eta_1^2\eta_2^2\eta_3\eta_4^2\eta_5\ringofintegers+
			\eta_1\eta_2\eta_3\eta_4\eta_6\eta_8\ringofintegers+
			\eta_8\eta_9\ringofintegers = C_0^2C_2^{-1}C_3^{-1}C_5^{-1}.
		\end{equation*}
		This proves the lemma.
	\end{proof}

	For $\bm{C}\in\mathcal{C}_i$, $i=1,2,3$, define
	\begin{align}\label{def_uC_i}
		\begin{split}
			u_{\mathbf{C},1}&= 
			\idealnorm{C_3^2C_4C_5},\\
			u_{\mathbf{C},2}&= 
			\idealnorm{C_3^2C_5},\quad\text{and}\\
			u_{\mathbf{C},3}&= 
			\idealnorm{C_1^2C_2^{-1}C_3^{-1}C_4^2C_5^{-1}}.
		\end{split}
	\end{align}
	
	The sets $M_i$ of sections define adelic metrics on the line bundles that are isomorphic to $\omega_{\widetilde{S}}(D_i)^\vee$,  $i=1,2,3$. Then, log-anticanonical height functions $\widetilde{H}_i$ are induced by these metrics for $i=1,2,3$
	(see for example \cite{pey95,pey03} on how heights are induced by metrics).
	
	\begin{lemma}\label{height_function_on_desingularisation}
		For $\bm{\eta} = (\eta_1,\ldots,\eta_9)\in K^9$ satisfying the torsor equation \eqref{torsor-equation_new} and condition \eqref{uni-condition} 
		let
		\begin{equation*}
			\mathcal{H}_i(\bm{\eta}) = 
			\begin{cases}
				\max\lbrace \norminf{\eta_4^2\eta_5\eta_7},
				\norminf{\eta_5\eta_6^2\eta_8},
				\norminf{\eta_4\eta_5^2\eta_6},
				\norminf{\eta_4\eta_6\eta_7\eta_8}\rbrace, & i=1,\\
				\max\lbrace \norminf{\eta_5\eta_7}, \norminf{\eta_5^2\eta_6},\norminf{\eta_6\eta_7\eta_8}\rbrace, & i=2,\\
				\max\lbrace \norminf{ \eta_1^2\eta_2^2\eta_3\eta_4^2\eta_5},\norminf{ \eta_1\eta_2\eta_3\eta_4\eta_6\eta_8},\norminf{\eta_8\eta_9}\rbrace, & i=3.
			\end{cases}
		\end{equation*}
		For $B\ge0$, $\bm{C}\in\mathcal{C}_i$ and $\bm{\eta}\in{}_{\bm{C}}\mathcal{Y}_{i}(\ringofintegers)\cap(\Psi^{-1}(V))(K)$,
		we have 
		\begin{equation*}
			\mathcal{H}_i(\bm{\eta})\le u_{\bm{C},i}B \text{ if and only if }\widetilde{H}_i(\rho(\bm{\eta})) = H_i(\pi(\rho(\bm{\eta})))\le B,
		\end{equation*}
		where $H_i$ is one of the height functions defined in \eqref{height_function_general}, \eqref{height_function_2} and \eqref{height_function_3}, and $\widetilde{H}_i$ is the log-anti\-ca\-no\-ni\-cal height on $\widetilde{S}(K)$ induced by the sections in \Cref{gcd_sets_M}, $i=1,2,3$.
	\end{lemma}
	\begin{proof}
		For $i\in\{1,2,3\}$, the above introduced log-anticanonical height functions induced by the metrics are given by $\widetilde{H}_i(x)=H_{\PP^{N_i}}(f_i(x))$ where $f_i(\rho(\bm{\eta})) = (m_0(\bm{\eta}):\dots:m_{N_i}(\bm{\eta}))$ for the sections $m_0,\ldots,m_{N_i}\in M_i$ constructed in \Cref{gcd_sets_M}.
		Therefore, 
		\begin{equation*}
			\widetilde{H}_i(\rho(\bm{\eta})) = \prod_{v\in\Omega_K}\max_{m\in M_i}\{\absvalue{m(\bm{\eta})}_v^2\}.
		\end{equation*}
		By definition of the $\mathfrak{p}$-adic absolute value, \Cref{gcd_sets_M}, and \eqref{def_uC_i}, the product over all prime ideals $\mathfrak{p}$ contributes the factor $u_{\bm{C},i}^{-1}$. Hence,
		\begin{equation*}
			\widetilde{H}_i(\rho(\bm{\eta})) = u_{\bm{C},i}^{-1} \max_{m\in M_i}\{ \norminf{m(\bm{\eta})}\} = u_{\bm{C},i}^{-1}\mathcal{H}_i(\bm{\eta}).
		\end{equation*}
		
		These height functions coincide with the ones defined in the introduction: For example, for a point $\bm{\eta}$ in ${}_{\bm{C}}\mathcal{Y}_{i}(\ringofintegers)\cap(\Psi^{-1}(V))(K)$ we have $\eta_1,\eta_2,\eta_3\in\ringofintegers^\times$. Thus $\norminf{\eta_4^2\eta_5\eta_7} = \norminf{\eta_1^2\eta_2^2\eta_3\eta_4^2\eta_5\eta_7} = \absvalue{x_0}^2$. We have analogous identities for the other coordinates and cases. 
		In addition, due to \Cref{gcd_sets_M} and \eqref{def_uC_i}, we have $\idealnorm{x_0\ringofintegers + \dots x_3\ringofintegers} = u_{\bm{C},1}$, and we obtain analogous results for the other two cases.
	\end{proof}

	\subsection{The Counting Problem}
	In this subsection we combine the results from the previous subsections to give a parameterisation of integral points on $\mathcal{U}_i$ via integral points on a universal torsor. We use this parameterisation to finally concretise our counting problem. 
	But first, we prove that there is a bijection between the integral points on the del Pezzo surface $S$ and its desingularisation $\widetilde{S}$. Therefore, it makes no difference to speak about integral points on the former in place of the latter. 
	
	\begin{lemma}\label{lemma_bijection_U_and_tilde(U)}
		For $i\in\{1,2,3\}$, the morphism $\pi\colon\widetilde{S}\rightarrow S$ induces bijections
		\[\widetilde{\mathcal{U}}_i(\ringofintegers)\cap\widetilde{V}(K)\rightarrow \mathcal{U}_i(\ringofintegers)\cap V(K).\]
		\begin{proof}
			We consider the morphism $f\colon\mathcal{Y}\rightarrow \PP^4_{\ringofintegers}$ which is given by $\bm{\eta}\mapsto(s_0(\bm{\eta}):\cdots:s_4(\bm{\eta}))$ where the $s_j$ are the anticanonical sections given in \eqref{sections_L0}. We already know that $\pi|_{\widetilde{V}}$ is an isomorphism. Hence, it induces a bijection between the sets $\widetilde{V}(K)$ and $V(K)$, and it remains to show that the integrality condition \eqref{I_j=O_K} resp. \eqref{uni-condition} for $\bm{\eta}$ is satisfied if and only if the corresponding integrality condition on $\mathcal{U}_i(\ringofintegers)$ is satisfied for $f(\bm{\eta})$.
			We note that \eqref{I_j=O_K} holds if and only if \eqref{uni-condition} holds with the choice of $\bm{C}$ and $\bm{C}'$ we make in \Cref{lemma_nicer_representation}. Hence, it suffices to consicer \eqref{I_j=O_K} here.
			
			Let us recall that a point $\bm{x} = (x_0:\cdots:x_4)$ lies in $\mathcal{U}_i(\ringofintegers)$ if and only if we can choose $x_0,\dots,x_4\in\ringofintegers$, and \eqref{corresponding_gcd_ideal} holds as well as \eqref{integrality condition} resp. \eqref{integrality_condition_2} resp. \eqref{gcd_cond_i=3} for $i=1$ resp. $i=2$ resp. $i=3$.
			
			Let $\bm{\eta}\in \widetilde{\mathcal{U}}_i(\ringofintegers)\cap \widetilde{V}(K)$. Then, by the \Cref{gcd_sets_M} and the definition of $I_j$ we have 
			\begin{align*}
				s_0(\bm{\eta})\ringofintegers + \dots + s_3(\bm{\eta})\ringofintegers &= \eta_1\eta_2\eta_3\sum_{m\in M_1}m(\bm{\eta})\ringofintegers = I_1I_2I_3C_0^{3}C_1^{-1}\cdots C_5^{-1},\\
				s_0(\bm{\eta})\ringofintegers + s_2(\bm{\eta})\ringofintegers + s_3(\bm{\eta})\ringofintegers &= \eta_1\eta_2\eta_3\eta_4\sum_{m\in M_2} m(\bm{\eta})\ringofintegers = I_1\cdots I_4 C_0^{3}C_1^{-1}\cdots C_5^{-1},\quad\text{and}\\
				s_0(\bm{\eta})\ringofintegers + s_3(\bm{\eta})\ringofintegers+s_4(\bm{\eta})\ringofintegers &= \eta_7\sum_{m\in M_3} m(\bm{\eta})\ringofintegers = I_7 C_0^{3}C_1^{-1}\cdots C_5^{-1},
			\end{align*}
			where $M_1,\,M_2$ and $M_3$ are defined in \Cref{gcd_sets_M}.
			One easily shows with \eqref{gcd-condition_eta} (see also \cite[Lemma 9.1]{derenthal14_compos}) that
			\[s_0(\bm{\eta})\ringofintegers+\cdots +s_4(\bm{\eta})\ringofintegers = C_0^{3}C_1^{-1}\cdots C_5^{-1}.\] 
			Hence, we have to show that \eqref{I_j=O_K} is equivalent to 
			\begin{align}\label{eq:prod_I_equals_OK}
				\begin{split}
					\begin{cases}
						I_1I_2I_3 = \ringofintegers & \quad\text{ if }i=1,\\
						I_1I_2I_3I_4 = \ringofintegers & \quad\text{ if }i=2,\\
						I_7 = \ringofintegers & \quad\text{ if }i=2.
					\end{cases}
				\end{split}
			\end{align}
			Clearly, \eqref{I_j=O_K} implies \eqref{eq:prod_I_equals_OK}. Vice versa, 
			note that $I_j\subseteq \ringofintegers$. Therefore, it is easy to see that the opposite direction also holds.
		\end{proof}
	\end{lemma}
	
	\begin{cor}\label{cor_correspondence_integral_points_universal_torsor_points}
		For $i\in\{1,2,3\}$, there is a $1$-to-$\omega_K^6$ correspondence between the set $\mathcal{U}_i(\ringofintegers)\cap V(K)$ of integral points and
		\[\bigcup_{\bm{C}\in\mathcal{C}_i} \{\mathbf{\eta}\in \mathcal{O}_{1*}\times\dots\times\mathcal{O}_{9*}\mid \eqref{torsor-equation_new},~ \eqref{gcd-condition_eta}, ~\eqref{uni-condition}\}.\]
		\begin{proof}
			This immediately follows from \Cref{lemma_parameterisation_of_integer_points_via_universal_torsor}, \Cref{lemma_nicer_representation}, and \Cref{lemma_bijection_U_and_tilde(U)}.
		\end{proof}
	\end{cor}
	
	Now, we can concretise our counting problem. As mentioned in \Cref{remark: i=3 not countable}, from now on we only consider $i=1,2$.
	Let us recall the definition of the height function $\mathcal{H}_i$ in \Cref{height_function_on_desingularisation} and the definition of $u_{\bm{C},i}$ in \eqref{def_uC_i}. For $i=1,2$ and $\bm{C}\in\mathcal{C}_i$, define 
	\begin{equation*}
		M_{\bm{C},i}(B) = \left\lbrace
		(\eta_1,\ldots,\eta_9)\in \mathcal{O}_{1*}\times\dots\times\mathcal{O}_{9*} ~
		\begin{array}{|c}
			\eqref{gcd-condition_eta}, ~\eqref{uni-condition},\\
			\eta_1\eta_4^2\eta_7+\eta_3\eta_6^2\eta_8+\eta_5\eta_9 = 0,\\
			\mathcal{H}_i(\eta_1,\ldots,\eta_8)\le u_{\bm{C},i}B        
		\end{array}
		\right\rbrace.
	\end{equation*}
	
	\begin{prop}\label{prop: N_i(B)}
		For $i\in\{1,2\}$ we obtain
		\begin{equation*}
			N_i(B) = \frac{1}{\omega_K^6}\sum_{\bm{C}\in\mathcal{C}_i}\# M_{\bm{C},i}(B).
		\end{equation*}
	\end{prop}
	\begin{proof}
		By combining \Cref{cor_correspondence_integral_points_universal_torsor_points}
		and \Cref{height_function_on_desingularisation}, this result is obtained similarly to \cite[Lemma 9.1]{derenthal14_compos} and \cite[Lemma 15]{derenthal_wilsch_21}. 
		Since $\# D_i$ components $C_j$ of $\bm{C}$ are uniquely determined by the remaining $C_j$'s, we only need to sum over $\bm{C}\in\mathcal{C}_i$.
	\end{proof}
	
	\begin{remark}
		As in \cite{derenthal14_compos}, it is also possible to obtain the result stated in \Cref{cor_correspondence_integral_points_universal_torsor_points} by using an elementary approach. \cite[Lemma 9.1]{derenthal14_compos} gives a $1$-to-$\omega_K^6$ correspondence between $V(K)$ and 
		\begin{equation*}
			\bigcup_{(C_0,\ldots,C_5)\in\mathcal{C}^6}\lbrace (\eta_1,\ldots,\eta_9)\in\mathcal{O}_{1^*}\times\cdots\times\mathcal{O}_{9^*}\mid \eqref{torsor-equation_new},\phantom{.} \eqref{gcd-condition_eta}\rbrace.
		\end{equation*}
		Then, \Cref{cor_correspondence_integral_points_universal_torsor_points} follows from similar arguments as in the proof of \Cref{lemma_parameterisation_of_integer_points_via_universal_torsor}.
	\end{remark}

	\section{Summations}\label{section: summations}
	To compute the number $N_i(B)$, using its representation in \Cref{prop: N_i(B)}, we split up the set $M_{\bm{C},i}(B)$ into two disjoint sets depending on the sizes of $\eta_7$ and $\eta_8$. In the first set, we sum first over the bigger variable $\eta_8$; in the second set we sum first over the bigger variable $\eta_7$.
	
	More concretely, we define $M^{(8)}_{\bm{C},i}(B)$ to be the set of $(\eta_1,\ldots,\eta_9)\in M_{\bm{C},i}(B)$ with $\idealnorm{I_8}\ge \idealnorm{I_7}$, and define $M^{(7)}_{\bm{C},i}(B)$ to be the set of $(\eta_1,\ldots,\eta_9)\in M_{\bm{C},i}(B)$ with $\idealnorm{I_7}>\idealnorm{I_8}$.
	Further, let \[N_{8,i}(B) = \frac{1}{\omega_K^6} \sum_{\bm{C}\in\mathcal{C}_i}\# M_{\bm{C},i}^{(8)}(B).\]
	We define $N_{7,i}(B)$ analogously.
	Then, clearly $N_i(B) = N_{8,i}(B) + N_{7,i}(B)$.
	
	From now on, we use the notation
	\begin{align*}
		\bm{\eta}^{(i)} &= (\eta_j)_{j\in J_i} = 
		\begin{cases}
			(\eta_4,\ldots,\eta_7), & i=1,\\
			(\eta_5,\eta_6,\eta_7), & i=2,
		\end{cases}\\
		\bm{I}^{(i)} &= (I_j)_{j\in J_i} = 
		\begin{cases}
			(I_4,\ldots,I_7), & i=1,\\
			(I_5,I_6,I_7), & i=2,
		\end{cases}
	\end{align*}
	and
	\begin{equation*}
		\bm{\mathcal{O}}_*^{(i)} = 
		\begin{cases}
			\mathcal{O}_{4^*}\times\dots\times\mathcal{O}_{7^*}, & i=1,\\
			\mathcal{O}_{5^*}\times\dots\times\mathcal{O}_{7^*}, & i=2,\\
		\end{cases}
	\end{equation*}
	for $(7-\#D_i)$-tuples indexed by 
	\begin{equation*}
		J_i = \{ j\in\{1,\ldots,7\}\mid E_j\not\subset D_i\}. 
	\end{equation*}
	We write $\idealnorm{\bm{I}^{(i)}} = (\idealnorm{I_j})_{j\in J_i}$ and $\mathcal{H}_i(\bm{\eta}^{(i)},\eta_8)$ for $\mathcal{H}_i(\eta_1,\ldots,\eta_9)$. Here, by using the torsor equation \eqref{torsor-equation_new} the variable $\eta_9$ is expressed in terms of $\eta_1,\ldots,\eta_8$, assuming $\eta_5\neq 0$, and $\eta_j\in\ringofintegers^\times$ whenever $E_j\subset D_i$.
	
	\subsection{The first summation}
	We start by summing over $\eta_8$ in $M_{\bm{C},i}^{(8)}(B)$ with dependent $\eta_9$. Due to the torsor equation \eqref{torsor-equation_new}, $\eta_9$ is dependent on $\eta_1,\ldots,\eta_8$. The rough idea is to estimate the sum over $\eta_8$ by an integral over the same region.
	Similarly to Lemma 9.2 in \cite{derenthal14_compos} we obtain
	
	\begin{lemma}\label{lemma_first_summation}
		For $B>0$, $\bm{C}\in\mathcal{C}_i$, $i=1,2$, we have
		\begin{align*}
			\# M_{\bm{C},1}^{(8)}(B) &= \frac{2\omega_K^3}{\sqrt{\lvert\Delta_K\rvert}} 
			\sum_{\bm{\eta}^{(1)}\in \mathcal{O}_{*}^{(1)}} \Theta_8(\bm{I}^{(1)}) V_8(\idealnorm{\bm{I}^{(1)}};B)
			+ O_{\bm{C}}(B\left(\log B\right)^3),\quad\text{and}\\
			\# M_{\bm{C},2}^{(8)}(B) &= \frac{2\omega_K^4}{\sqrt{\absvalue{\Delta_K}}}
			\sum_{\bm{\eta}^{(2)}\in \mathcal{O}_{*}^{(2)}} \Theta_8(\bm{I}^{(2)}) V_8(\idealnorm{\bm{I}^{(2)}};B)
			+ O_{\bm{C}}(B\left(\log B\right)),
		\end{align*}
		where \begin{equation*}
			V_8(\bm{t}^{(i)};B) = \frac{1}{t_5}\int_{\substack{\mathcal{H}_i((\sqrt{t_j})_{j\in J_i},\eta_8)\le B\\ \norminf{\eta_8} \ge t_7}} \mathrm{d}\eta_8
		\end{equation*}
		with a complex variable $\eta_8$, and where
		\begin{equation*}
			\Theta_8(\bm{I}^{(i)})= \prod_\mathfrak{p}\Theta_{8,\mathfrak{p}}(I_\mathfrak{p}(\bm{I}^{(i)}))
		\end{equation*}
		with $I_\mathfrak{p}(\bm{I}^{(i)}) = \lbrace j\in J_i :\mathfrak{p}\divides I_j\rbrace$ and
		\begin{equation*}
			\Theta_{8,\mathfrak{p}}(I) =  
			\begin{dcases}
				1 & \text{if }I = \emptyset,\{5\},\{6\},\{7\},\\
				1-\frac{1}{\idealnorm{\mathfrak{p}}} & \text{if } I=\{1\},\{3\},\{4\},\{1,2\},\{1,4\},\{2,3\},\{2,5\},\{3,6\},\{4,7\},\\
				1-\frac{2}{\idealnorm{\mathfrak{p}}} & \text{if }I = \{2\},\\
				0 & \text{otherwise}.
			\end{dcases}
		\end{equation*}
		Further, if we replace the condition $\idealnorm{I_8}\ge\idealnorm{I_7}$ in the definition of $M_{\bm{C}}^{(8)}$ by $\idealnorm{I_8}>\idealnorm{I_7}$, the same asymptotic formula holds.
	\end{lemma}
	\begin{proof}
		The proof for the main term is analogous to Lemma 9.2 in \cite{derenthal14_compos}. As we have slightly different height functions and some $\eta_i\in\ringofintegers^\times$, we obtain different error terms. 
		
		We notice that in the case $i=1$, we have $I_j = \ringofintegers$ for $j=1,2,3$. Thus, $\idealnorm{I_j} = 1$ for $j=1,2,3$. 
		Further, there are no prime ideals dividing $I_j$ for $j=1,2,3.$ Hence, nothing is dependent on $\eta_1,\eta_2,\eta_3$ and the sum over these $\eta_j$ yields the factor $\omega_K^3$. Similarly, in the case $i=2$, the sum over $\eta_1,\dots,\eta_4$ yields the factor $\omega_K^4$.
		
		Now, we compute the error terms. We start with the case $i=1$. Due to the fourth height condition in \Cref{height_function_on_desingularisation}, we have $\norminf{\eta_8} \le u_{\bm{C},1}\frac{B}{\norminf{\eta_4\eta_6\eta_7}}$. Hence, the set $\mathcal{R}_1(\bm{\eta}^{(1)};u_{\bm{C},1}B)\subseteq\CC$ of $\eta_8$ with $\mathcal{H}(\bm{\eta}^{(1)},\eta_8)\le u_{\bm{C},1}B$ and $\idealnorm{I_8}\ge \idealnorm{I_7}$ is contained in a ball of radius
		\[ R_1(\bm{\eta}^{(1)};u_{\bm{C},1}B)=  u_{\bm{C},1}^{1/2}=\frac{B^{1/2}}{\norminf{\eta_4\eta_6\eta_7}^{1/2}} \ll_{\bm{C}} \frac{B^{1/2}}{(\idealnorm{I_4}\idealnorm{I_6}\idealnorm{I_7})^{1/2}}.\]
		Therefore, the error term is (see also \cite[Lemma 9.2]{derenthal14_compos})
		\begin{equation*}
			\ll_{\bm{C}}
			\sum_{\bm{\eta}^{(1)}\in\mathcal{O}_{*}^{(1)}}
			2^{\omega(I_4)}
			\left( 
			\frac{B^{1/2}}{\idealnorm{I_4}^{1/2}\idealnorm{I_5}^{1/2}\idealnorm{I_6}^{1/2}\idealnorm{I_7}^{1/2}} +1
			\right),
		\end{equation*}
		where $\omega(I_j)$ denotes the number of distinct prime divisors of $I_j$.
		Similar to \cite{derenthal14_compos} we can replace the sums over $\eta_j\in\mathcal{O}_{j^*}$ by sums over the ideals $I_j\in\mathcal{I}_K$, since there are at most $\lvert \ringofintegers^\times\rvert<\infty$ elements $\eta_j\in\mathcal{O}_j$ with $I_j=\mathfrak{a}$ for an ideal $\mathfrak{a}\in\mathcal{I}_K$.
		Therefore, the error term is
		\begin{equation*}
			\ll_{\bm{C}}
			\sum_{\bm{I}^{(1)}\in \mathcal{I}_K^4}
			2^{\omega(I_4)}\
			\left( 
			\frac{B^{1/2}}{\idealnorm{I_4}^{1/2}\idealnorm{I_5}^{1/2}\idealnorm{I_6}^{1/2}\idealnorm{I_7}^{1/2}} +1
			\right).
		\end{equation*}
		We first sum over $I_6$. Due to the second height condition in \Cref{height_function_on_desingularisation} and our assumption $\idealnorm{I_8}\ge \idealnorm{I_7}$ we have 
		\begin{equation*}
			\idealnorm{I_6}\le \frac{B^{1/2}}{\idealnorm{I_5}^{1/2}\idealnorm{I_8}^{1/2}}\le \frac{B^{1/2}}{\idealnorm{I_5}^{1/2}\idealnorm{I_7})^{1/2}}.
		\end{equation*}
		Hence, the error term becomes
		\begin{align*}
			&\ll_{\bm{C}}  \sum_{I_4,I_5,I_7\in \mathcal{I}_K}
			2^{\omega(I_4)}\
			\left( 
			\frac{B^{3/4}}{\idealnorm{I_4}^{1/2}\idealnorm{I_5}^{3/4}\idealnorm{I_7}^{3/4}} +
			\frac{B^{1/2}}{\idealnorm{I_5}^{1/2}\idealnorm{I_7}^{1/2}}
			\right)\\
			&\ll_{\bm{C}} \sum_{I_4,I_5\in \mathcal{I}_K}
			2^{\omega(I_4)}\
			\left( 
			\frac{B}{\idealnorm{I_4}\idealnorm{(I_5}} +
			\frac{B}{\idealnorm{I_4}\idealnorm{I_5}}
			\right)\\
			&\ll_{\bm{C}} B\log B\sum_{I_4\in\mathcal{I}_K}
			\frac{2^{\omega(I_4)}}{\idealnorm{I_4}},
		\end{align*}
		where we summed over $I_7$ in the second line with $\idealnorm{I_7}\le \frac{B}{\idealnorm{I_4}^2\idealnorm{I_5}}.$
		Lemma 2.4 and 2.9 in \cite{derenthal14_compos} yield 
		\[\sum_{I_4\in\mathcal{I}_K} \frac{2^{\omega(I_4)}}{\idealnorm{I_4}}\ll \left(\log B\right)^2.\]
		We finally obtain that the error term is $\ll_{\bm{C}} B\left(\log B\right)^3.$
		
		For the case $i=2$, the set $\mathcal{R}_2(\bm{\eta}^{(2)}; u_{\bm{C},2}B)\subset\CC$ of $\eta_8$ with $\mathcal{H}_2(\bm{\eta}^{(2)})\le u_{\bm{C},2}B$ and $\idealnorm{I_8}\ge\idealnorm{I_7}$ is contained in a ball of radius
		\[R_2(\bm{\eta}^{(2)}; u_{\bm{C},2}B) = u_{\bm{C},2}^{1/2}\frac{B^{1/2}}{\norminf{\eta_6\eta_7}^{1/2}} \ll_{\bm{C}} \frac{B^{1/2}}{(\idealnorm{I_6}\idealnorm{I_7})^{1/2}}.\]
		Hence, we obtain that the error term is
		\begin{equation*}
			\ll_{\bm{C}} \sum_{\bm{\eta}^{(2)}\in\mathcal{O}_{*}^{2}}
			\left( 
			\frac{B^{1/2}}{\idealnorm{I_5}^{1/2}\idealnorm{I_6}^{1/2}\idealnorm{I_7}^{1/2}} +1
			\right).
		\end{equation*}
		As in the previous case, we can sum over the ideals $I_j\in \mathcal{I}_K$ instead, since $\absvalue{\ringofintegers^\times}<\infty$. Thus, the error term is
		\begin{equation*}
			\ll_{\bm{C}} \sum_{\bm{I}^{(2)}\in \mathcal{I}_K^3}
			\left( 
			\frac{B^{1/2}}{\idealnorm{I_5}^{1/2}\idealnorm{I_6}^{1/2}\idealnorm{I_7}^{1/2}} +1
			\right).
		\end{equation*}
		We first sum over $I_7$. We use the third height condition in \Cref{height_function_on_desingularisation} and $\idealnorm{I_8}\ge \idealnorm{I_7}$.
		Then, the error term becomes
		\begin{equation*}
			\ll_{\bm{C}} \sum_{I_5,I_6\in\mathcal{I}_K} \left(\frac{B^{3/4}}{\idealnorm{I_5}^{1/2}\idealnorm{I_6}^{3/4}} + \frac{B^{1/2}}{\idealnorm{I_6}^{1/2}}\right)
			\ll_{\bm{C}} \sum_{I_6\in\mathcal{I}_K}\left(
			\frac{B}{\idealnorm{I_6}} + \frac{B}{\idealnorm{I_6}}\right)
			\ll_{\bm{C}} B\log B.
		\end{equation*}
		This proves the lemma.
	\end{proof}
	The next step is to replace the sums over $\eta_i$ by sums over the corresponding ideals $I_i$.
	
	\begin{lemma}\label{lemma_transformation_eta_into_ideals}
		For $i\in\lbrace 1,2\rbrace$, we have
		\[N_{8,i}(B) = \omega_Kh_K^{-1}\frac{2}{\sqrt{\lvert\Delta_K\rvert}}
		\sum_{\bm{I}^{(i)}} \Theta_8(\bm{I}^{(i)})V_8(\idealnorm{\bm{I}^{(i)}};B) + O(B\left(\log B\right)^{d_i}),\]
		where $d_1 = 3$, $d_2=1$, and the sum runs over all $(6-\#D_i)$-tuples of non-zero ideals of $\ringofintegers$.
	\end{lemma}
	\begin{proof}
		Let us recall that $\mathcal{C} = \{P_1,\dots,P_{h_K}\}$. For $i=1,2$, by the definition of $N_{8,i}(B)$ and \Cref{lemma_first_summation} we have
		\begin{align*}
			N_{8,i}(B) &= \frac{1}{\omega_K^{6-\# D_i}}\frac{2}{\sqrt{\lvert \Delta_K\rvert}}
			\sum_{\bm{C}\in\mathcal{C}_i} \sum_{\bm{\eta}^{(i)}\in\bm{\mathcal{O}}_*^{(i)}}
			\Theta_8(\bm{I}^{(i)})V_8(\idealnorm{\bm{I}^{(i)}};B)\\
			&+ \sum_{\bm{C}\in\mathcal{C}_i}O_{\bm{C}}\left(B(\log B)^{d_i}\right)\\
			&= \frac{1}{\omega_K^{6-\#D_i}}\frac{2}{\sqrt{\lvert \Delta_K\rvert}}
			\sum_{j=1}^{h_K} \sum_{\substack{\bm{C}\in\mathcal{C}_i\\ [\mathcal{O}_{7^*}] = [P_j]}} \sum_{\bm{\eta}^{(i)}\in\bm{\mathcal{O}}_*^{(i)}}
			\Theta_8(\bm{I}^{(i)})V_8(\idealnorm{\bm{I}^{(i)}};B)\\
			&+O_{\bm{C}}\left(B(\log B)^{d_i}\right).
		\end{align*}
		In the inner sum, $\eta_7$ runs through all nonzero elements of $\mathcal{O}_7$. This implies that $I_7$ runs through all ideals $\neq 0$ in the ideal class of $P_j$, each ideal occurring $\omega_K$ times.
		We can bound the ideal norm of any of the ideals $I_7$ by $B$, due to the height conditions occurring in $V_8$.
		Hence,
		\begin{align*}
			N_{8,i}(B) = \frac{1}{\omega_K^{6-\#D_i-1}}\frac{2}{\sqrt{\lvert \Delta_K\rvert}} \sum_{j=1}^{h_K}
			\sum_{\substack{I_7\in [P_j]\\ \idealnorm{I_7}\le B}} 
			\sum_{\substack{\bm{C}\in\mathcal{C}_i\\ [\mathcal{O}_{7^*}] = [P_j]}} \sum_{\substack{(\eta_{3+i},\dots,\eta_6)\\\in\mathcal{O}_{(3+i)^*}\times\cdots\times\mathcal{O}_{6^*}}}
			\Theta_8(\bm{I}^{(i)})V_8(\idealnorm{\bm{I}^{(i)}};B)\\
			+O\left(B(\log B)^{d_i}\right).
		\end{align*}
		The sum over all $I_7\in [P_j]$ with $\idealnorm{I_7}\le B$ is independent on the choice of the ideal class $P_j$. Therefore, we can replace this sum by $h_K^{-1}\sum_{\substack{I_7\in \mathcal{I}_K\\ \idealnorm{I_7}\le B}} $ and obtain
		\begin{equation*}
			N_{8,i}(B) = \frac{2h_K^{-1}}{\omega_K^{6-\#D_i-1}\sqrt{\lvert \Delta_K\rvert}}
			\sum_{\substack{I_7\in \mathcal{I}_K\\ \idealnorm{I_7}\le B}} 
			\sum_{\bm{C}\in\mathcal{C}_i} \sum_{\substack{(\eta_{3+i},\dots,\eta_6)\\\in\mathcal{O}_{(3+i)^*}\times\cdots\times\mathcal{O}_{6^*}}} \Theta_8(\bm{I}^{(i)})V_8(\idealnorm{\bm{I}^{(i)}};B)
			+O\left(B(\log B)^{d_i}\right).
		\end{equation*}
		An analogous argument as in the proof of Lemma 9.4 in \cite{derenthal14_compos} yields the lemma for $i=1,2$.
	\end{proof}

	\subsection{The remaining summations}
	\begin{lemma}\label{lemma_remaining_summations}
		For $B>0$ we have
		\begin{align*}
			N_{8,1}(B) &= \frac{2}{\sqrt{\lvert\Delta_K\rvert}}\omega_K\rho_K^4h_K^{-1}
			\Theta_0^{(1)}
			V_{80}^{(1)}(B)
			+ O\left(B\left(\log B\right)^3\log\log B\right),\quad\text{and}\\
			N_{8,2}(B) &= \frac{2}{\sqrt{\lvert\Delta_K\rvert}}\omega_K\rho_K^3h_K^{-1}
			\Theta_0^{(2)}V_{80}^{(2)}(B)
			+ O\left(B\left(\log B\right)^{2}\log\left(\log B\right)\right),
		\end{align*}
		where 
		\begin{equation*}
			V_{80}^{(i)}(B) = \int_{1\le t_j\le B~\forall 1\le j<8 \text{ with } E_j\not\subset D_i} V_8(\bm{t}^{(i)};B)\mathrm{d}\bm{t}^{(i)},
		\end{equation*}
		and
		\begin{align}
			\Theta_0^{(1)} &= \prod_\mathfrak{p} \left( 1-\frac{1}{\idealnorm{\mathfrak{p}}}\right)^3\left(1+\frac{3}{\idealnorm{\mathfrak{p}}}\right),\quad\text{and} \label{Theta_01}\\
			\Theta_0^{(2)} &= \prod_\mathfrak{p} \left( 1 - \frac{1}{\idealnorm{\mathfrak{p}}}\right)^2\left( 1 + \frac{2}{\idealnorm{\mathfrak{p}}}\right), \label{Theta_02}
		\end{align}
		where the product runs over all prime ideals $\mathfrak{p}.$
	\end{lemma}
	\begin{proof}
		Under certain assumptions on the main term, \cite[Proposition 7.3]{derenthal14_compos} gives us a tool to handle the summations over the remaining variables at once. We begin with checking the necessary precondition on the main term.
		The fourth height condition for $i=1$ and the third height condition in the case $i=2$ in \Cref{height_function_on_desingularisation} yield
		\[V_8(\bm{t}^{(1)};B) \ll \frac{B}{t_4t_5t_6t_7} \text{ and } V_8(\bm{t}^{(2)};B) \ll \frac{B}{t_5t_6t_7}.\]
		Hence, the condition for $V$ in \cite[Section 7]{derenthal14_compos} for the case (a) with $s=0$ is satisfied. With an analogous argumentation to the cases (b) and (c) we obtain an analogous result to \cite[Proposition 7.3]{derenthal14_compos} for $s=0$, where we have to replace $r-1$ by $r$ in the exponent in the error term. With $s=0$ and $r=4$ in the case $i=1$, or $s=0$ and $r=3$ in the case $i=2$,respectively, we obtain the first part of the lemma.
		
		It remains to compute $\Theta_0^{(i)}$. Therefore, we use Lemma 2.8 in \cite{derenthal14_compos} for $l=1$:
		\begin{align*}
			\Theta_0^{(1)} &= \mathcal{A}(\Theta_8(\bm{I}^{(1)}),I_7,\ldots,I_4)\\
			&= \prod_{\mathfrak{p}}\sum_{L\subset\{4,5,6,7\}} \left(1-\frac{1}{\idealnorm{\mathfrak{p}}}\right)^{4-\lvert L\rvert}
			\left(\frac{1}{\idealnorm{\mathfrak{p}}}\right)^{\lvert L\rvert}\Theta_{8,\mathfrak{p}}(L)\\
			&= \prod_\mathfrak{p}\left(\left(1-\frac{1}{\idealnorm{\mathfrak{p}}}\right)^4 + 3\cdot\left(1-\frac{1}{\idealnorm{\mathfrak{p}}}\right)^3\cdot\frac{1}{\idealnorm{\mathfrak{p}}} \right.\\
			&\phantom{=} \left.+ \left(1-\frac{1}{\idealnorm{\mathfrak{p}}}\right)^3\cdot\frac{1}{\idealnorm{\mathfrak{p}}}\cdot\left(1-\frac{1}{\idealnorm{\mathfrak{p}}}\right)\right.\\
			&\phantom{=} \left.+\left(1-\frac{1}{\idealnorm{\mathfrak{p}}}\right)^2\cdot\left(\frac{1}{\idealnorm{\mathfrak{p}}}\right)^2\cdot\left(1-\frac{1}{\idealnorm{\mathfrak{p}}}\right)\right)\\
			&= \prod_\mathfrak{p} \left( 1-\frac{1}{\idealnorm{\mathfrak{p}}}\right)^3\left(1+\frac{3}{\idealnorm{\mathfrak{p}}}\right),
		\end{align*}
		and
		\begin{align*}
			\Theta_0^{(2)} &= \mathcal{A}(\Theta_8(\bm{I}^{(2)}),I_7,\ldots,I_5)\\
			&= \prod_{\mathfrak{p}}\sum_{L\subset\{5,6,7\}} \left(1-\frac{1}{\idealnorm{\mathfrak{p}}}\right)^{3-\lvert L\rvert}
			\left(\frac{1}{\idealnorm{\mathfrak{p}}}\right)^{\lvert L\rvert}\Theta_{8,\mathfrak{p}}(L)\\
			&= \prod_{\mathfrak{p}}\left( 1-\frac{1}{\idealnorm{\mathfrak{p}}}\right)^3 + 3\cdot\left(1-\frac{1}{\idealnorm{\mathfrak{p}}}\right)^2\frac{1}{\idealnorm{\mathfrak{p}}}\\
			&= \prod_{\mathfrak{p}}\left( 1- \frac{1}{\idealnorm{\mathfrak{p}}}\right)^2\left( 1+\frac{2}{\idealnorm{\mathfrak{p}}}\right).
		\end{align*}
		This completes the proof.
	\end{proof}
	
	We can use symmetries to compute $N_{7,i}(B)$. This allows us to combine the results for $N_{8,i}(B)$ and $N_{7,i}(B)$ to obtain a result for $N_i(B)$.
	
	\begin{prop}\label{lemma:N_i(B)_with_V_0}
		We have
		\begin{align*}
			N_1(B) &=  \left(\frac{2}{\sqrt{\lvert\Delta\rvert}}\right)^5\frac{1}{\omega_K^3}h_K^3 
			\Theta_0^{(1)} V_0^{(1)}(B)
			+O\left(B\left(\log B\right)^3\log(\log B)\right),\quad\text{and}\\
			N_2(B) &=  \left(\frac{2}{\sqrt{\lvert\Delta\rvert}}\right)^4\frac{1}{\omega_K^2}h_K^2 
			\Theta_0^{(2)} V_0^{(2)}(B)
			+O\left(B\left(\log B\right)^2\log(\log B)\right),
		\end{align*}
		where $\Theta_0^{(i)}$ for $i=1,2$ is given in \eqref{Theta_01}-\eqref{Theta_02} and
		\begin{align*}
			V_0^{(i)}(B) = \int\displaylimits_{\substack{\norminf{\eta_j}\ge 1~\forall 1\le j\le 8\text{ with }E_j\not\subset D_i,\\ 
					\mathcal{H}_i(\bm{\eta}^{(i)},\eta_8)\le B}}
			\norminf{\eta_5}^{-1}\mathrm{d}\eta_4\cdots\mathrm{d}\eta_8.
		\end{align*}
	\end{prop}
	\begin{proof}
		This proof works similarly as the proof of Lemma 9.9 in \cite{derenthal14_compos}. Let \[J^{(i)} = \lbrace j\in\{1,\ldots,8\}\mid E_j\not\subset D_i\rbrace.\]
		Using the substitution $\sqrt{t_j} = r_j$ for $j\in J^{(i)}$ and subsequently polar coordinates, we obtain 
		\begin{equation*}
			V_{80}^{(i)}(B) = \pi^{-(7-\#D_i)}\int\displaylimits_{\substack{\norminf{\eta_j}\ge 1~\forall j\in J^{(i)},\\\norminf{\eta_8}\ge \norminf{\eta_7} \\ 
					\mathcal{H}_i(\bm{\eta}^{(i)},\eta_8)\le B}}
			\norminf{\eta_5}^{-1}\mathrm{d}\eta_4\cdots\mathrm{d}\eta_8.
		\end{equation*}
		Therefore, 
		\begin{equation*}
			N_{8,i}(B) = \left(\frac{2}{\sqrt{\lvert\Delta_K\rvert}}\right)^{6-\# D_i+2}\left(\frac{h_K}{\omega_K}\right)^{6-\#D_i}
			\pi^{7-\#D_i}\Theta_0^{(i)}V_{80}^{(i)}(B) \\+ O\left(B\left(\log B\right)^{6-\#D_i}\log\log B\right).
		\end{equation*}
		Now, we compute $N_{7,i}(B)$. Here, we consider $\norminf{\eta_8} <\norminf{\eta_7}$.  We start by summing over $\eta_7$. For $i=1$, due to symmetry reasons, by changing the variable numbers $6\leftrightarrow4$, $7\leftrightarrow8$, $1\leftrightarrow3$ we can perform the first summation over $\eta_7$ analogously to \Cref{lemma_first_summation} and the remaining summations analogously to \Cref{lemma_remaining_summations}. We obtain 
		\begin{equation*}
			N_{7,1}(B) = \left(\frac{2}{\sqrt{\lvert\Delta_K\rvert}}\right)^5\frac{1}{\omega_K^3}h_K^3
			\pi^4\Theta_0^{(1)}V_{70}^{(1)}(B) + O\left(B\left(\log B\right)^3\log\log B\right),
		\end{equation*}
		where
		\begin{equation*}
			V_{70}^{(1)}(B) = \pi^{-4}\int\displaylimits_{\substack{\norminf{\eta_j}\ge 1~\forall j\in J^{(1)},\\ \norminf{\eta_8}\le \norminf{\eta_7} \\ 
					\mathcal{H}_i(\bm{\eta}^{(1)},\eta_8)\le B}}
			\norminf{\eta_5}^{-1}\mathrm{d}\eta_4\cdots\mathrm{d}\eta_8.
		\end{equation*}
		In the case $i=2$, we make the same change of variables as in the case $i=1$. Here, we need to make some adjustments, since now $\eta_6$ is a unit in $\ringofintegers$. For the first summation, we again proceed like in \Cref{lemma_first_summation} with a slightly different error term:
		The set $\mathcal{R}'_2(\eta_4,\eta_5,\eta_7;u_{\bm{C},2}B)\subset\CC$ of $\eta_8$ with $\mathcal{H}_2(\eta_4,\eta_5,\eta_7,\eta_8)\le u_{\bm{C},2}B$ and $\idealnorm{I_8}\ge \idealnorm{I_7}$ is contained in a ball of radius
		\[R_2'(\eta_4,\eta_5,\eta_7; u_{\bm{C},2}) = u_{\bm{C},2}^{1/2}\frac{B^{1/2}}{\norminf{\eta_4\eta_7}^{1/2}}\ll_{\bm{C}} \frac{B^{1/2}}{(\idealnorm{I_4}\idealnorm{I_7})^{1/2}}.\]
		Thus, the error term is
		\begin{equation*}
			\ll_{\bm{C}} \sum_{(\eta_4,\eta_5,\eta_7)\in\mathcal{O}_{4*}\times\mathcal{O}_{5*}\times\mathcal{O}_{7*}} 2^{\omega(I_4)}\left( \frac{B^{1/2}}{\idealnorm{I_4}^{1/2}\idealnorm{I_5}^{1/2}\idealnorm{I_7}^{1/2}}+1\right).
		\end{equation*}
		We can sum over the ideals $I_j\in\mathcal{I}_K$ instead, since $\absvalue{\ringofintegers^\times}<\infty$. Thus, the error term is
		\begin{equation*}
			\ll_{\bm{C}}\sum_{I_4,I_5,I_7\in\mathcal{I}_K}2^{\omega(I_4)}\left( \frac{B^{1/2}}{\idealnorm{I_4}^{1/2}\idealnorm{I_5}^{1/2}\idealnorm{I_7}^{1/2}}+1 \right).
		\end{equation*}
		We first sum over $I_5$ by using the second height condition in \Cref{height_function_on_desingularisation}. Then, the error terms is
		\begin{equation*}
			\ll_{\bm{C}} \sum_{I_4,I_7\in\mathcal{I}_K}2^{\omega(I_4)} \left( \frac{B^{3/4}}{\idealnorm{I_4}^{3/4}\idealnorm{I_7}^{1/2}}+\frac{B^{1/2}}{\idealnorm{I_4}^{1/2}}\right)
			\ll_{\bm{C}} \sum_{I_4\in\mathcal{I}_K} 2^{\omega(I_4)}\frac{B}{\idealnorm{I_4}}
			\ll_{\bm{C}} B \log B^2.
		\end{equation*}
		In the second estimation we used the third height condition in \Cref{height_function_on_desingularisation} and $\idealnorm{I_8}\ge \idealnorm{I_7}$. In the last estimation, we used \cite[Lemmas 2.4 and 2.9]{derenthal14_compos}. The remaining summations work similarly to \Cref{lemma_remaining_summations}.
		The lemma follows.
	\end{proof}

	\subsection{Computing $V_0^{(i)}(B)$}
	The next step is to replace the integral $V_0^{(i)}$ by another integral $V_0^{(i)'}$, which then turns out to be the product of a constant, the volume of a polytope, and $B(\log B)^{6-\#D_i}$, $i=1,2$. To this end, for $i=1,2$, define
	\begin{align*}
		&\mathcal{R}_0^{(i)}(B) = \lbrace (\bm{\eta}^{(i)},\eta_8) ~\mid~ \norminf{\eta_j}\ge 1 ~\forall j\rbrace,\\
		&\mathcal{R}_1^{(1)}(B) = \lbrace (\bm{\eta}^{(1)},\eta_8) ~\mid~ \norminf{\eta_j}\ge 1 ~\forall j, \norminf{\eta_4\eta_6\eta_8}\le B\rbrace,\\
		&\mathcal{R}_2^{(1)}(B) = \left\lbrace (\bm{\eta}^{(1)},\eta_8) ~\mid~ \norminf{\eta_j}\ge 1 ~\forall j, \norminf{\eta_4\eta_6\eta_8}\le B, \frac{\norminf{\eta_4\eta_5}}{\norminf{\eta_6\eta_8}}\le 1\right\rbrace,\\
		&\mathcal{R}_3^{(1)}(B) = \left\lbrace (\bm{\eta}^{(1)},\eta_8) ~\mid~ \norminf{\eta_j}\ge 1 ~\forall j\neq 7, \norminf{\eta_4\eta_6\eta_8}\le B, \frac{\norminf{\eta_4\eta_5}}{\norminf{\eta_6\eta_8}}\le 1\right\rbrace,\\
		&\mathcal{R}_1^{(2)}(B) = \lbrace (\bm{\eta}^{(2)},\eta_8) ~\mid~ \norminf{\eta_j}\ge 1 ~\forall j, \norminf{\eta_6\eta_8}\le B\rbrace,\\
		&\mathcal{R}_2^{(2)}(B) = \left\lbrace (\bm{\eta}^{(2)},\eta_8) ~\mid~ \norminf{\eta_j}\ge 1 ~\forall j, \norminf{\eta_6\eta_8}\le B, \frac{\norminf{\eta_5}}{\norminf{\eta_6\eta_8}}\le 1\right\rbrace,\quad\text{and}\\
		&\mathcal{R}_3^{(2)}(B) = \left\lbrace (\bm{\eta}^{(2)},\eta_8) ~\mid~ \norminf{\eta_j}\ge 1 ~\forall j\neq 7, \norminf{\eta_6\eta_8}\le B, \frac{\norminf{\eta_5}}{\norminf{\eta_6\eta_8}}\le 1\right\rbrace.
	\end{align*}
	For simplicity, let $\widetilde{\bm{\eta}}^{(i)} = (\bm{\eta}^{(i)},\eta_8)$.
	Define
	\begin{align*}
		V^{(i,j)}(B) &= \int\displaylimits_{\substack{\mathcal{H}_i(\widetilde{\bm{\eta}}^{(i)})\le B,\\ \widetilde{\bm{\eta}}^{(i)}\in\mathcal{R}_j^{(i)}(B)}}\norminf{\eta_5}^{-1}\mathrm{d}\widetilde{\bm{\eta}}^{(i)},\quad\text{and}\\
		V_0^{(i)'}(B) &= \int\displaylimits_{\substack{\mathcal{H}_i(\widetilde{\bm{\eta}}^{(i)})\le B,\\ \widetilde{\bm{\eta}}^{(i)}\in\mathcal{R}_3^{(i)}(B)}}\norminf{\eta_5}^{-1}\mathrm{d}\widetilde{\bm{\eta}}^{(i)}.
	\end{align*}
	
	\begin{lemma}\label{lemma:change_the_integration_area_of_V_0}
		For $B>0$ we have
		\begin{equation*}
			V_0^{(1)}(B) = V^{(1)'}_0(B) + O\left(B(\log B)^3\right),\text{ and}\quad
			V_0^{(2)}(B) = V^{(2)'}_0(B) + O\left(B(\log B)^2\right).
		\end{equation*}
		\begin{proof}
			It suffices to prove 
			\begin{align*}
				\left\lvert V^{(1,j)}(B)-V^{(1,j+1)}(B)\right\rvert &\ll B(\log B)^3,\quad\text{and}\\
				\left\lvert V^{(2,j)}(B)-V^{(2,j+1)}(B)\right\rvert &\ll B(\log B)^2
			\end{align*}
			for all $0\le j\le 2$, 
			as $V_0^{(i)}(B) = V^{(i,0)}(B)$, $V_0^{(i)'}(B) = V^{(i,3)}(B)$ and
			\begin{equation*}
				V_0^{(i)}(B) = V_0^{(i)}(B)-V^{(i,1)}(B)+V^{(i,1)}(B)-V^{(i,2)}(B)+V^{(i,2)}(B)-V^{(i,3)}(B)+V^{(i,3)}(B)
			\end{equation*}
			for $i=1,2$.
			At first, we notice that for $0\le j\le 2$, $i=1,2$, we have
			\begin{equation*}
				\left\lvert V^{(i,j)}(B)-V^{(i,j+1)}(B)\right\rvert \ll
				\int\displaylimits_{\substack{\widetilde{\bm{\eta}}^{(i)}\in(\mathcal{R}_j^{(i)}(B)\cup \mathcal{R}_{j+1}^{(i)}(B))\setminus (\mathcal{R}_j^{(i)}(B)\cap \mathcal{R}_{j+1}^{(i)}(B)),\\ \mathcal{H}_i(\widetilde{\bm{\eta}}^{(i)})\le B}}
				\norminf{\eta_5}^{-1}\mathrm{d}\widetilde{\bm{\eta}}^{(i)}.
			\end{equation*}
			
			We start with considering the case $i=1$.
			Let $j=0$: The fourth height condition and $\norminf{\eta_7}\ge 1$ imply $\norminf{\eta_4\eta_6\eta_8}\le B$. Hence, $V^{(1,0)}(B)=V^{(1,1)}(B)$.
			
			Let $j=1$: As $\mathcal{R}_2^{(1)}(B)\subseteq \mathcal{R}_1^{(1)}(B)$, we consider $\widetilde{\bm{\eta}}^{(1)}\in\mathcal{R}_1^{(1)}(B)\setminus\mathcal{R}_2^{(1)}(B)$.
			We have $\norminf{\eta_8} < \frac{\norminf{\eta_4\eta_5}}{\norminf{\eta_6}}$ (complement of the condition that we add in $\mathcal{R}_2^{(1)}(B)$), $\norminf{\eta_4^2}\le\frac{B}{\norminf{\eta_5\eta_7}}$ (first height condition) as well as $1\le\norminf{\eta_5}, \norminf{\eta_6},\norminf{\eta_7}\le B$ (this follows from the height conditions).
			This yields
			\begin{equation*}
				\int\displaylimits_{\substack{\mathcal{H}_i(\widetilde{\bm{\eta}}^{(1)})\le B\\ \widetilde{\bm{\eta}}^{(1)}\in \mathcal{R}_1^{(1)}(B)\setminus \mathcal{R}_2^{(1)}(B)}}
				\frac{1}{\norminf{\eta_5}}\mathrm{d}\widetilde{\bm{\eta}}^{(1)}
				\ll \int \frac{\norminf{\eta_4}}{\norminf{\eta_6}}\mathrm{d}\eta_4\cdots\mathrm{d}\eta_7
				\ll \int \frac{B}{\norminf{\eta_5\eta_6\eta_7}}\mathrm{d}\eta_5\mathrm{d}\eta_6\mathrm{d}\eta_7
				\ll B(\log B)^3,
			\end{equation*}
			where we integrated over $\eta_8$ in the first step and $\eta_4$ in the second step.
			
			Let $j=2$: As $\mathcal{R}_2^{(1)}(B)\subseteq \mathcal{R}_3^{(1)}(B)$, we consider $\widetilde{\bm{\eta}}^{(1)}\in \mathcal{R}_3^{(1)}(B)\setminus\mathcal{R}_2^{(1)}(B)$. 
			We have $\norminf{\eta_7}< 1$ (complement of the condition we remove in $\mathcal{R}_3^{(1)}$), $\norminf{\eta_4}\le \frac{B}{\norminf{\eta_6\eta_8}}$ and $1\le \norminf{\eta_5},\norminf{\eta_6},\norminf{\eta_8}\le B$ (this follows from the height conditions). 
			We obtain
			\begin{align*}
				\int\displaylimits_{\substack{\mathcal{H}_i(\widetilde{\bm{\eta}}^{(1)})\le B\\ \widetilde{\bm{\eta}}^{(1)}\in\mathcal{R}_3^{(1)}(B)\setminus\mathcal{R}_2^{(1)}(B)}}
				\frac{1}{\norminf{\eta_5}}\mathrm{d}\widetilde{\bm{\eta}}^{(1)}
				&\ll \int \frac{1}{\norminf{\eta_5}}\mathrm{d}\eta_4\mathrm{d}\eta_5\mathrm{d}\eta_6\mathrm{d}\eta_8\\
				&\ll \int \frac{B}{\norminf{\eta_5\eta_6\eta_8}} \mathrm{d}\eta_5\mathrm{d}\eta_6\mathrm{d}\eta_8\\
				&\ll B\left(\log B\right)^3,
			\end{align*}
			where we integrated over $\eta_7$ in the first step and $\eta_4$ in the second step.
			
			Now consider $i=2$. The case $j=0$ works analogously to the case above. The remaining two cases are similar, too, with different error terms:
			For $j=1$ we have
			\begin{equation*}
				\int\displaylimits_{\substack{\mathcal{H}_i(\widetilde{\bm{\eta}}^{(2)})\le B\\ \widetilde{\bm{\eta}}^{(2)}\in\mathcal{R}_1^{(1)}(B)\setminus\mathcal{R}_2^{(1)}(B)}}
				\frac{1}{\norminf{\eta_5}}\mathrm{d}\widetilde{\bm{\eta}}^{(2)}
				\ll \int \frac{1}{\norminf{\eta_6}}\mathrm{d}\eta_5\cdots\mathrm{d}\eta_7
				\ll B\int\frac{1}{\norminf{\eta_6\eta_7}}\mathrm{d}\eta_6\mathrm{d}\eta_7
				\ll B(\log B)^2,
			\end{equation*}
			where we used $\norminf{\eta_8}<\frac{\norminf{\eta_5}}{\norminf{\eta_6}}$ (complement of the condition that we add in $\mathcal{R}_2^{(2)}$) and $\norminf{\eta_5}\le \frac{B}{\norminf{\eta_7}}$ (first height condition).
			For $j=2$ we get 
			\begin{align*}
				\int\displaylimits_{\substack{\mathcal{H}_i(\widetilde{\bm{\eta}}^{(2)})\le B\\ \widetilde{\bm{\eta}}^{(2)}\in\mathcal{R}_3^{(1)}(B)\setminus\mathcal{R}_2^{(1)}(B)}}
				\frac{1}{\norminf{\eta_5}}\mathrm{d}\widetilde{\bm{\eta}}^{(2)}
				&\ll \int \frac{1}{\norminf{\eta_5}}\mathrm{d}\eta_5\mathrm{d}\eta_6\mathrm{d}\eta_8\\
				&\ll B\int\frac{1}{\norminf{\eta_5\eta_6}}\mathrm{d}\eta_5\mathrm{d}\eta_6\\
				&\ll B(\log B)^2
			\end{align*}
			by using $\norminf{\eta_7}<1$ and $\norminf{\eta_8}\le \frac{B}{\norminf{\eta_6}}$.
		\end{proof}
	\end{lemma}
	
	Next, we show that we can remove the first and the third height condition in $V_0^{(1)'}$ and the first height condition in $V_0^{(2)'}$.
	
	\begin{cor} \label{cor_V_0}
		For $B>0$ we have
		\begin{align*}
			V_0^{(1)}(B) &= \int\displaylimits_{\substack{\norminf{\eta_i}\ge 1~\forall i\neq7\\ \norminf{\eta_4\eta_6\eta_8}\le B\\
					\norminf{\eta_4\eta_5}/\norminf{\eta_6\eta_8}\le 1\\        \norminf{\eta_5\eta_6^2\eta_8}\le B\\
					\norminf{\eta_4\eta_6\eta_7\eta_8}\le B}}
			\norminf{\eta_5}^{-1}\mathrm{d}\eta_4\cdots\mathrm{d}\eta_8
			+O\left(B\left(\log B\right)^3\right),\quad\text{and}\\
			V_0^{(2)}(B) &= \int\displaylimits_{\substack{\norminf{\eta_i}\ge 1~\forall i\neq7\\
					\norminf{\eta_6\eta_8}\le B\\
					\norminf{\eta_5}/\norminf{\eta_6\eta_8}\le 1\\
					\norminf{\eta_5^2\eta_6}\le B\\
					\norminf{\eta_6\eta_7\eta_8}\le B}}\norminf{\eta_5}^{-1}\mathrm{d}\eta_5\cdots\mathrm{d}\eta_8 + O\left( B\left(\log B\right)^2\right).
		\end{align*}
		\begin{proof}
			In the first case, with the property $\norminf{\eta_4\eta_5}\le \norminf{\eta_6\eta_8}$ we obtain
			\[\norminf{\eta_4^2\eta_5\eta_7} \le \norminf{\eta_4\eta_6\eta_7\eta_8}\le B. \]
			Therefore, the first height condition is redundant.
			The same property also yields
			\[\norminf{\eta_4\eta_5^2\eta_6} \le \norminf{\eta_5\eta_6^2\eta_8} \le B.\]
			Hence, the third height condition is redundant, too.
			Similarly to the first case, we obtain
			\[\norminf{\eta_5\eta_7}\le \norminf{\eta_6\eta_7\eta_8}\le B\]
			by using $\norminf{\eta_5}\le \norminf{\eta_6\eta_8}$.
			This makes the first height condition redundant in the second case.
			Combining these results with the previous lemma completes the proof.
		\end{proof}
	\end{cor}

	\begin{prop}\label{def_C_i}
		For $B>0$ we have
		\begin{align*}
			V_0^{(1)}(B) &= \frac{\pi^3}{4}\constant_1\cdot B\log B^4  + O\left(B(\log B)^3\right),\quad\text{and}\\
			V_0^{(2)}(B) &= \frac{\pi^2}{4}\constant_2\cdot B\log B^3  + O\left(B(\log B)^2\right)
		\end{align*}
		with
		\begin{align*}
			\constant_1 &= 4\pi^2\vol \left\lbrace (t_4,t_5,t_6,t_8)\in\RR^4_{\ge0}~
			\begin{array}{|c}
				t_5+2t_6+t_8\le 1,\\
				t_4+t_6+t_8\le 1,\\
				t_4+t_5-t_6-t_8\le 0
			\end{array}
			\right\rbrace=\frac{4\pi^2}{72}=\frac{\pi^2}{18},\quad\text{and}\\
			\constant_2 &= 4\pi^2\vol \left\lbrace (t_5,t_6,t_8)\in\RR^3_{\ge0}~
			\begin{array}{|c}
				t_6+t_8\le 1,\\
				t_5-t_6-t_8\le 0,\\
				2t_5+t_6\le 1
			\end{array} \right\rbrace = \frac{4\pi^2\cdot 11}{72} = \frac{11\pi^2}{18}.
		\end{align*}
	\end{prop}
	\begin{proof}
		We use the representation of $V_0^{(1)}(B)$ in the previous \Cref{cor_V_0}. Integrating over $\eta_7$ by using $\norminf{\eta_7}\le \frac{B}{\norminf{\eta_4\eta_6\eta_8}}$ yields
		\begin{equation*}
			V_0^{(1)}(B)=
			\pi\int\displaylimits_{\substack{\norminf{\eta_i}\ge 1~\forall i\neq7\\ 
					\norminf{\eta_4\eta_6\eta_8}\le B\\
					\norminf{\eta_4\eta_5}/\norminf{\eta_6\eta_8}\le 1\\
					\norminf{\eta_5\eta_6^2\eta_8}\le B}}
			\frac{B}{\norminf{\eta_4\eta_5\eta_6\eta_8}}\mathrm{d}\eta_4\mathrm{d}\eta_5\mathrm{d}\eta_6\mathrm{d}\eta_8
			+O\left(B(\log B)^3\right).
		\end{equation*}
		Then, we change the variables into polar coordinates, i.e.\ $\eta_j = r_j e^{2\pi i\varphi_j}$.
		Subsequently, we substitute $r_j = \sqrt{s_j}$. We obtain 
		\begin{align*}
			V_0^{(1)}(B) &= \pi\cdot B \int_0^{2\pi}\mathrm{d}\varphi_4\mathrm{d}\varphi_5\mathrm{d}\varphi_6\mathrm{d}\varphi_8
			\int\displaylimits_{\substack{ r_i\ge 1~\forall i\neq 7\\
					\left( r_4r_6r_8\right)^2\le B\\
					r_4r_5/ r_6r_8\le 1\\
					\left( r_5r_6^2r_8\right)^2\le B}}
			\frac{1}{ r_4r_5r_6r_8}
			\mathrm{d}r_4\mathrm{d}r_5\mathrm{d}r_6\mathrm{d}r_8
			+O\left(B(\log B)^3\right)\\
			&= \pi^5\cdot B\cdot \frac{2^4}{2^4}
			\int\displaylimits_{\substack{s_i\ge 1~\forall i\neq 7\\
					s_4s_6s_8\le B\\
					s_4s_5/s_6s_8\le 1\\
					s_5s_6^2s_8\le B}}
			\frac{1}{\left(s_4s_5s_6s_8\right)^{1/2}}\left(s_4s_5s_6s_8\right)^{-1/2}\mathrm{d}s_4\mathrm{d}s_5\mathrm{d}s_6\mathrm{d}s_8
			+O\left(B(\log B)^3\right)\\
			&= \pi^5\cdot B 
			\int\displaylimits_{\substack{s_i\ge 1~\forall i\neq 7\\
					s_4s_6s_8\le B\\
					s_4s_5/s_6s_8\le 1\\
					s_5s_6^2s_8\le B}}
			\frac{1}{s_4s_5s_6s_8}\mathrm{d}s_4\mathrm{d}s_5\mathrm{d}s_6\mathrm{d}s_8
			+O\left(B(\log B)^3\right).
		\end{align*}
		Finally, by substituting $s_i = B^{t_i}$, we obtain 
		\begin{align*}
			V_0^{(1)}(B) &= \pi^5\cdot B\log B^4\int\displaylimits_{\substack{t_i\ge 0\\
					t_5+2t_6+t_8\le 1\\
					t_4+t_6+t_8\le 1\\
					t_4+t_5-t_6-t_8\le 0}}
			\mathrm{d}t_4\mathrm{d}t_5\mathrm{d}t_6\mathrm{d}t_8
			+O\left(B\log B^3\right)\\
			&= \frac{\pi^3}{4}\cdot B\log B^4\constant_1+O\left(B\log B^3\right).
		\end{align*}
		The proof for $V_0^{(2)}(B)$ works similar.
	\end{proof}
	
	This completes the proof of \Cref{main_theorem}.

	\section{The leading constant}\label{section: The leading constant}
	This section is based on section 6 in \cite{derenthal_wilsch_21}. We show that \eqref{N_i(B)_minimal_desing} holds, that means, that \Cref{main_theorem} can be reinterpreted in the framework for interpreting the asymptotic behaviour of the number of integral points of bounded height.
	For the finite part of the leading constant \eqref{const_fin}, we have to compute $\pprime$-adic Tamagawa volumes $\tau_{(\widetilde{S},D_i),\pprime}(\widetilde{\mathcal{U}}_i(\localringofintegers))$, which are defined in \cite[§§ 2.1.10, 2.4.3]{chambertloir_tschinkel_2010}. 
	In our case, these measures are similar to the usual Tamagawa volumes, which are studied in the context of rational points, except for factors $\norm{1_{D_i}}_\pprime$, which are constant and equal to 1 on the set of $\pprime$-adic integral points at all finite places.
	The analogous volumes over the archimedean places would be infinite, when we evaluate them on the full space of complex points. Instead, \emph{residue measures} $\tau_{i,D_A,\infty}$ supported on the minimal strata $D_A(\CC)$ of the boundary divisors show up in the leading constant \eqref{const_infty}, cf. \cite[§ 2.1.12]{chambertloir_tschinkel_2010}.
	We can interpret them as a density function for the set of integral points, cf. \cite[§ 3.5.8]{chambertloir_tschinkel_2012}, or the leading constant of an asymptotic expansion of the volume of \emph{height balls} with respect to $\tau_{(\widetilde{S},D_i),\infty}$, cf. \cite[Theorem 4.7]{chambertloir_tschinkel_2010}.
	
	Further, the rational factors $\alpha_{i,A}$, appearing in \eqref{const_infty}, have to be computed, one for each minimal stratum $A$ of the boundary $D_i$. 
	
	We start by computing the Tamagawa volumes. To this end, we work with the chart 
	\begin{align*}
		f\colon V'=\widetilde{S}\setminus V(\eta_1\eta_2\eta_3\eta_4\eta_5\eta_6)&\rightarrow \mathbb{A}_K^2\\
		(\eta_1:\eta_2:\eta_3:\eta_4:\eta_5:\eta_6:\eta_7:\eta_8:\eta_9)&\mapsto
		\left(\frac{\eta_4}{\eta_2\eta_3\eta_5\eta_6}\cdot\eta_7,\frac{\eta_6}{\eta_1\eta_2\eta_4\eta_5}\cdot\eta_8\right).
	\end{align*}
	Its inverse $g\colon\mathbb{A}_K^2\rightarrow\widetilde{S}$ is given by
	\begin{equation*}
		(x,y)\mapsto (1:1:1:1:1:1:x:y:-x-y).
	\end{equation*}
	An easy computation shows that the two elements \[\frac{\eta_4}{\eta_2\eta_3\eta_5\eta_6}\cdot\eta_7, \quad \text{ and }\quad\frac{\eta_6}{\eta_1\eta_2\eta_4\eta_5}\cdot\eta_8\]
	have degree $0$ in the field of fractions of the Cox ring. Thus, they define a rational map which is invariant under the torus action and descends to $\widetilde{S}$.
	
	\begin{lemma}
		For any prime ideal $\mathfrak{p}\subseteq\ringofintegers$, the images of the sets of $\mathfrak{p}$-adic integral points are
		\begin{align*}
			f(\widetilde{\mathcal{U}_1}(\mathcal{O}_{K,\mathfrak{p}})\cap V'(K_\mathfrak{p})) &= \lbrace (x,y)\in K_\mathfrak{p}^2\mid \pabsvalue{xy}\le 1 \text{ or }\pabsvalue{x+y}\le 1\rbrace,\quad\text{and}\\
			f(\widetilde{\mathcal{U}_2}(\mathcal{O}_{K,\mathfrak{p}})\cap V'(K_\mathfrak{p})) &= \lbrace (x,y)\in K_\mathfrak{p}^2\mid (\pabsvalue{y}\le 1\text{ and }\pabsvalue{xy}\le 1) \text{ or }\pabsvalue{x+y}\le 1\rbrace.
		\end{align*}
	\end{lemma}
	\begin{proof}
		Consider the image \[(x,y) = \left(\frac{\eta_4}{\eta_2\eta_3\eta_5\eta_6}\cdot\eta_7,\frac{\eta_6}{\eta_1\eta_2\eta_4\eta_5}\cdot\eta_8\right)\]
		of an integral point $\rho(\eta_1,\ldots,\eta_9)\in \widetilde{\mathcal{U}}_1(\localringofintegers)$ under $f$.
		We have
		\begin{equation*}
			x\cdot y= \frac{\eta_7\eta_8}{\eta_1\eta_2^2\eta_3\eta_5^2}, \quad\text{ and }\quad
			x+y = \frac{\eta_1\eta_4^2\eta_7+\eta_3\eta_6^2\eta_8}{\eta_1\eta_2\eta_3\eta_4\eta_5\eta_6} = \frac{-\eta_5\eta_9}{\eta_1\cdots\eta_6}.
		\end{equation*}
		Since $\eta_1,\eta_2,\eta_3\in\ringofintegers^\times$ on $\widetilde{\mathcal{U}}_1$, we obtain
		\begin{equation*}
			\lvert xy\rvert_\mathfrak{p} = \left\lvert\frac{\eta_7\eta_8}{\eta_5^2}\right\rvert_\mathfrak{p},
			\quad\text{ and }\quad
			\lvert x+y\rvert_\mathfrak{p} = \left\lvert\frac{\eta_9}{\eta_4\eta_6}\right\rvert_\mathfrak{p}
		\end{equation*}
		for all prime ideals $\mathfrak{p}.$
		
		Assume $\lvert xy\rvert_\mathfrak{p}>1$. Then, $\eta_5\not\in\mathcal{O}_{K,\mathfrak{p}}^\times.$ Due to the coprimality conditions in \Cref{fig:1-Dynkin-diagram}, we get $\eta_i\in\mathcal{O}_{K,\mathfrak{p}}^\times$ for all $i=1,\ldots,4,6\ldots,8$. This yields
		\[\lvert x+y\rvert_\mathfrak{p} = \lvert\eta_9\rvert_\mathfrak{p} \le 1.\]
		
		On the other hand, let us consider a point $(x,y)\in K_\mathfrak{p}^2$ with $\lvert xy\rvert_\mathfrak{p}\le 1$ or $\lvert x+y\rvert_\mathfrak{p}\le 1$. We want to construct an integral point $(\eta_1,\ldots,\eta_9)$ on the torsor with $f(\rho(\eta_1,\ldots,\eta_9)) = (x,y)$.
		
		If $\lvert xy\rvert_\mathfrak{p}\le 1$, we distinguish three cases:
		\begin{enumerate}
			\item If $\lvert x\rvert_\mathfrak{p}\le 1$ and $\lvert y\rvert_\mathfrak{p}\le 1$, let $\eta_7=x$, $\eta_8 = y$, $\eta_9 = x-y$, and the remaining coordinates be $1$. Obviously, the coprimality conditions are satisfied. Further, we have $f(\rho(\eta_1,\ldots,\eta_9)) = (x,y)$ and the torsor equation is satisfied.
			
			\item If $\lvert x\rvert_\mathfrak{p}\le 1$ and $\lvert y\rvert_\mathfrak{p}>1$, let $\eta_4 = 1/y$, $\eta_7=xy$, $\eta_9 = -1-x/y$, and the remaining coordinates be $1$.
			Since $\left\lvert\frac{x}{y}\right\rvert_\mathfrak{p}\le \left\lvert\frac{1}{y}\right\rvert_\mathfrak{p}<1$, we have $\eta_9\in-1+\mathfrak{p}\mathcal{O}_{K,\mathfrak{p}}\subseteq \mathcal{O}_{K,\mathfrak{p}}^\times$ and the coprimality conditions hold.
			
			\item If $\lvert x\rvert_\mathfrak{p}>1$ and $\lvert y\rvert_\mathfrak{p}\le 1$, let $\eta_6 = 1/x$, $\eta_8 = xy$, $\eta_9 = -1-y/x$, and the remaining coordinates be $1$.
			Since $\left\lvert\frac{y}{x}\right\rvert_\mathfrak{p}\le \lvert\frac{1}{x}\rvert_\mathfrak{p}<1$, we have $\eta_9\in -1+\mathfrak{p}\mathcal{O}_{K,\mathfrak{p}}\subseteq\mathcal{O}_{K,\mathfrak{p}}^\times,$
			and thus the coprimality conditions are satisfied.
		\end{enumerate}
		If $\lvert xy\rvert_\mathfrak{p}>1$, we have $\lvert x\rvert_\mathfrak{p}>1$ and $\lvert y\rvert_\mathfrak{p}>1$: Assume $\lvert x\rvert_\mathfrak{p}\le1$ and $\lvert y\rvert_\mathfrak{p}>1$. Then, we obtain 
		\[\lvert x+y\rvert_\mathfrak{p} = \max\lbrace\lvert x\rvert_\mathfrak{p},\lvert y\rvert_\mathfrak{p}\rbrace = \lvert y\rvert_\mathfrak{p}>1,\]
		which is a contradiction to $\lvert x+y\rvert_\mathfrak{p}\le 1$.
		The same argument works for $\pabsvalue{x}>1$ and $\pabsvalue{y}\le1$.
		Let $\eta_5 = 1/y$, $\eta_9 = -x-y$, $\eta_7 = -\eta_5\eta_9-1$, and the remaining coordinates be $1$.
		Then, 
		\[\frac{\eta_7}{\eta_5} = (-\eta_5\eta_9-1)\cdot\frac{1}{\eta_5} = -\eta_9-\frac{1}{\eta_5} = x+y-y=x.\] 
		Hence, $f(\rho(\eta_1,\ldots,\eta_9))=(x,y).$ Further, as $(-\eta_5\eta_9-1)+1+\eta_5\eta_9=0$, the torsor equation holds.
		It remains to check whether $\eta_7\in\mathcal{O}_{K,\mathfrak{p}}^\times$.
		It is $\pabsvalue{\eta_9}=\absvalue{x+y}\le1.$ Hence, $\pabsvalue{\eta_5\eta_9}
		\le \pabsvalue{\eta_5}<1.$ Therefore, $\eta_7\in -1+\mathfrak{p}\mathcal{O}_{K,\mathfrak{p}}\subseteq \mathcal{O}_{K,\mathfrak{p}}^\times$ and the coprimality conditions are satisfied.
		
		Now, let $(x,y)$ be the image of an integral point $\rho(\eta_1,\ldots,\eta_9)\in\widetilde{\mathcal{U}}_2(\localringofintegers)$. If $\pabsvalue{y}>1$, we have $\eta_5\not\in\localringofintegers^\times$. Due to the coprimality conditions in \Cref{fig:1-Dynkin-diagram}, we obtain $\eta_i\in\localringofintegers^\times$ for all $i=1,\ldots,4,6,\ldots,8$. Thus, $\pabsvalue{x+y} = \pabsvalue{\eta_9}\le 1$. Moreover, we obtain $\pabsvalue{xy}=\frac{1}{\pabsvalue{\eta_5}^2}>1$. Analogously, we have $\eta_5\not\in\localringofintegers^\times$ if $\pabsvalue{xy}>1$, and the same argument shows $\pabsvalue{y}>1$.
		
		Vice versa, let $(x,y)\in K_\pprime^2$ with ($\pabsvalue{y}\le 1$ and $\lvert xy\rvert_\mathrm{p}<1$) or $\pabsvalue{x+y}\le 1$. We want to construct an integral point on the torsor lying above $(x,y)$. The two cases $\pabsvalue{y}\le1$ and $\pabsvalue{x}\le 1$, as well as $\pabsvalue{y}\le 1$ and $\pabsvalue{x}>1$ with the extra condition $\pabsvalue{xy}\le 1$ work as the above cases (1) and (3). 
		It remains to consider $\pabsvalue{y}>1$. As in the situation of $\widetilde{\mathcal{U}_1}(\mathcal{O}_{K,\mathfrak{p}})$, one shows that $\pabsvalue{x+y}\le1$ implies $\pabsvalue{x}>1$, too. Then, we can choose the same values for $\eta_i$, $i=1,\ldots,9$, as above.
	\end{proof} 
	
	\begin{lemma}
		Let $v$ be a place of $K$. We have
		\begin{align*}
			\mathrm{d}f_{*}\tau_{(\widetilde{S},D_1),v} &= \frac{1}{\max\lbrace 1,\absvalue{x}_v,\absvalue{y}_v,\absvalue{xy}_v \rbrace}\mathrm{d}x\,\mathrm{d}y,\quad\text{and}\\
			\mathrm{d}f_{*}\tau_{(\widetilde{S},D_2),v} &= \frac{1}{\max\lbrace1,\absvalue{x}_v,\absvalue{xy}_v\rbrace}\mathrm{d}x\,\mathrm{d}y
		\end{align*}
		for the measures $\tau_{(\widetilde{S},D_i),v}$ defined in \cite[§ 2.4.3]{chambertloir_tschinkel_2010}.
	\end{lemma}
	\begin{proof}
		In the first case, we have
		\begin{equation}\label{norm_D_1}
			\mathrm{d}f_{*}\tau_{(\widetilde{S},D_1),v} = \norm{(\mathrm{d}x\wedge\mathrm{d}y)\otimes 1_{E_1}\otimes 1_{E_2}\otimes 1_{E_3}}^{-1}_{\omega_{\widetilde{S}}(D_1),v}\mathrm{d}x\,\mathrm{d}y.
		\end{equation}
		Arguing as in \cite[Lemma 23]{derenthal_wilsch_21}, with $\mathrm{d}x\wedge\mathrm{d}y$ mapping to $1/\eta_1^2\eta_2^3\eta_3^2\eta_4\eta_5^2\eta_6$, the norm in \eqref{norm_D_1} at a point $\bm{\eta}$ can be written as
		\begin{equation}\label{norm_in_cox_coordinates_1}
			\frac{\norm{\eta_1^2\eta_2^3\eta_3^2\eta_4\eta_5^2\eta_6}_v}{\norm{\eta_1\eta_2\eta_3}_v\max\lbrace \norm{\eta_1\eta_2\eta_4^2\eta_5\eta_7}_v,\norm{\eta_2\eta_3\eta_5\eta_6^2\eta_8}_v,\norm{\eta_1\eta_2^2\eta_3\eta_4\eta_5^2\eta_6}_v,\norm{\eta_4\eta_6\eta_7\eta_8}_v \rbrace}.
		\end{equation}
		Evaluating this in the image $(x,y)$ of $f$ in $\bm{\eta}$ yields the statement.
		
		In the second case, we have
		\begin{equation*}
			\mathrm{d}f_{*}\tau_{(\widetilde{S},D_2),v} = 
			\norm{(\mathrm{d}x\wedge\mathrm{d}y)\otimes 1_{E_1}\otimes 1_{E_2}\otimes 1_{E_3}\otimes 1_{E_4} }^{-1}_{\omega_{\widetilde{S}}(D_2)}\mathrm{d}x\,\mathrm{d}y,
		\end{equation*}
		and analogously determine the norm of this in Cox coordinates:
		\begin{equation}\label{norm_in_cox_coordinates_2}
			\frac{\norm{\eta_1^2\eta_2^3\eta_3^2\eta_4\eta_5^2\eta_6}_v}{\norm{\eta_1\eta_2\eta_3\eta_4}_v\max\lbrace \norm{\eta_1\eta_2\eta_4\eta_5\eta_7}_v,\norm{\eta_1\eta_2^2\eta_3\eta_5^2\eta_6}_v,\norm{\eta_6\eta_7\eta_8}_v \rbrace}.
		\end{equation}
	\end{proof}
	
	\begin{prop}
		Let $\pprime$ be a prime ideal in $K$. We have
		\begin{equation*}
			\tau_{(\widetilde{S},D_i),\pprime}(\widetilde{\mathcal{U}}_i(\localringofintegers)) = 1+\frac{6-\#D_i}{\idealnorm{\pprime}}.
		\end{equation*}
	\end{prop}
	\begin{proof}
		We integrate $\mathrm{d}f_{*}\tau_{(\widetilde{S},D_i),\pprime}$ over the set of integral points $f(\widetilde{\mathcal{U}}_i(\localringofintegers) \cap V'(K_\pprime))$, that is
		\begin{equation}\label{tamagava_volume_formula}
			\tau_{(\widetilde{S},D_i),\pprime}(\widetilde{\mathcal{U}}_i(\localringofintegers)) = \int_{f(\widetilde{\mathcal{U}}_i(\localringofintegers) \cap V'(K_\pprime))}\mathrm{d}f_{*}\tau_{(\widetilde{S},D_i),\pprime}.
		\end{equation}
		With the two previous lemmas, for $i=1$ this converts into
		\begin{equation*}
			\int_{\substack{x,y\in K_\pprime\\ \pabsvalue{xy}\le 1\text{ or }\pabsvalue{x+y}\le1}}
			\frac{1}{\max\lbrace 1,\pabsvalue{x},\pabsvalue{y},\pabsvalue{xy} \rbrace}\mathrm{d}x\,\mathrm{d}y
		\end{equation*}
		for the Tamagawa volumes at finite places.
		
		We want to compute this volume. Therefore, we subdivide the domain of integration into the regions with
		$\pabsvalue{x},\pabsvalue{y}>1$, and $\pabsvalue{y}>1\ge\pabsvalue{x}$, and $\pabsvalue{x}>1\ge\pabsvalue{y}$, and $\pabsvalue{x},\pabsvalue{y}\le1$
		in order to simplify the denominator. We obtain
		\begin{equation}\label{split_integral_tau}
			\begin{split}
				\int_{\substack{\pabsvalue{x},\pabsvalue{y}>1\\ \pabsvalue{x+y}\le 1}} \frac{1}{\pabsvalue{xy}}\mathrm{d}x\,\mathrm{d}y
				+ \int_{\substack{\pabsvalue{xy}\le 1\\ \pabsvalue{y}>1\ge \pabsvalue{x}}}  \frac{1}{\pabsvalue{y}}\mathrm{d}x\,\mathrm{d}y
				+ \int_{\substack{\pabsvalue{xy}\le 1\\ \pabsvalue{x}>1\ge \pabsvalue{y}}}  \frac{1}{\pabsvalue{x}}\mathrm{d}x\,\mathrm{d}y
				+ \int_{\pabsvalue{x},\pabsvalue{y}\le 1} \mathrm{d}x\,\mathrm{d}y.
			\end{split}
		\end{equation}
		We start by computing the first integral in \eqref{split_integral_tau}. Due to $\pabsvalue{y}>1$ and $\pabsvalue{x+y}\le1$ we have $\pabsvalue{x+y}<\pabsvalue{y}$. Therefore, $\pabsvalue{x}=\pabsvalue{y}$. Hence, the integral simplifies to
		\begin{equation*}
			\int_{\pabsvalue{y}>1,\pabsvalue{x+y}\le 1}\frac{1}{\pabsvalue{y}^2}\mathrm{d}x\,\mathrm{d}y.
		\end{equation*}
		This can be transformed into
		\begin{align*}
			\int_{\pabsvalue{y}>1}\frac{1}{\pabsvalue{y}^2}\mathrm{d}y\int_{\pabsvalue{x}\le 1}\mathrm{d}x
			&= \int_{\pabsvalue{y}>1}\frac{1}{\pabsvalue{y}^2}\mathrm{d}y\\
			&= \sum_{k=-\infty}^{-1}\int_{\pabsvalue{y}=\idealnorm{\pprime}^{-k}}\frac{1}{\idealnorm{\pprime}^{-2k}}\mathrm{d}y\\
			&= \sum_{k=-\infty}^{-1}\idealnorm{\pprime}^{2k}\frac{1}{\idealnorm{\pprime}^k}\left(1-\frac{1}{\idealnorm{\pprime}}\right)\\
			&= \left(1-\frac{1}{\idealnorm{\pprime}}\right) \frac{1}{\idealnorm{\pprime}}\sum_{k=0}^\infty \idealnorm{\pprime}^{-k}\\
			&= \frac{1}{\idealnorm{\pprime}}.
		\end{align*}
		The second and third integral in \eqref{split_integral_tau} are symmetric, hence identical, and each has the value
		\begin{align*}
			\int_{\substack{\pabsvalue{y}>1,\pabsvalue{x}\le1\\ \pabsvalue{y}\le\frac{1}{\pabsvalue{x}}}} \frac{1}{\pabsvalue{y}}\mathrm{d}x\,\mathrm{d}y 
			&= \int_{\pabsvalue{x}\le 1}\left( \sum_{k=-v_\pprime(x)}^{-1} \int_{\pabsvalue{y}=\idealnorm{\pprime}^{-k}}\frac{1}{\idealnorm{\pprime}^{-k}}\mathrm{d}y \right)\mathrm{d}x\\
			&= \int_{\pabsvalue{x}\le 1}\left( \sum_{k=-v_\pprime(x)}^{-1} \idealnorm{\pprime}^{k}\frac{1}{\idealnorm{\pprime}^k}\left(1-\frac{1}{\idealnorm{\pprime}}\right) \right)\mathrm{d}x\\
			&= \int_{\pabsvalue{x}\le 1} \left(1-\frac{1}{\idealnorm{\pprime}}\right) v_\pprime(x)\mathrm{d}x\\
			&= \left(1-\frac{1}{\idealnorm{\pprime}}\right) \sum_{k=0}^\infty\int_{\pabsvalue{x}=\idealnorm{\pprime}^{-k}}k\mathrm{d}x\\
			&= \left(1-\frac{1}{\idealnorm{\pprime}}\right)^2 \sum_{k=0}^\infty\frac{k}{\idealnorm{\pprime}^k}\\
			&= \frac{1}{\idealnorm{\pprime}}.
		\end{align*}
		The fourth integral has the value $1$. Adding the four terms in \eqref{split_integral_tau} gives the claim for $i=1$.
		
		Let $i=2$. The previous lemmas transform \eqref{tamagava_volume_formula} into
		\begin{equation*}
			\int_{\substack{x,y\in K_\pprime,~\pabsvalue{x+y}\le1\text{ or }\\ (\pabsvalue{y}\le 1 \text{ and } \pabsvalue{xy}\le1)}}
			\frac{1}{\max\lbrace 1,\pabsvalue{x},\pabsvalue{xy} \rbrace}\mathrm{d}x\,\mathrm{d}y.
		\end{equation*}
		Again, we subdivide the domain of integration into smaller regions and obtain
		\begin{equation*}
			\int_{\substack{\pabsvalue{x},\pabsvalue{y}>1\\\pabsvalue{x+y}\le1}}\frac{1}{\pabsvalue{xy}}\mathrm{d}x\,\mathrm{d}y
			+ \int_{\substack{\pabsvalue{y}\le 1,\pabsvalue{x}>1\\\pabsvalue{xy}\le1}}\frac{1}{\pabsvalue{x}}\mathrm{d}x\,\mathrm{d}y
			+ \int_{\pabsvalue{x},\pabsvalue{y}\le 1}\mathrm{d}x\,\mathrm{d}y.
		\end{equation*}
		Analogous to the case $i=1$, we obtain $\frac{1}{\idealnorm{\pprime}}$ for the first and second integral, and 1 for the last integral. Summing these three values gives the statement for $i=2$.
	\end{proof}

	The remaining parts of the constant are archimedean measures, which are associated with maximal faces of the Clemens complex (see e.g.\ \cite[§ 3.1]{chambertloir_tschinkel_2010} for a definition).
	The divisor $D_1$ has three vertices corresponding to its components, and
	two $1$-simplices, which we will call $A_1=\{E_1,E_2\}$ and $A_2=\{E_2,E_3\}$, added between the intersecting exceptional curves (see \Cref{fig:Clemens-complex}).
	The Clemens complex of the divisor $D_2$ consists of four vertices corresponding to its components, and three $1$-simplices $A_1$, $A_2$ and $A_3 = \{E_3,E_4\}$ between the intersecting exceptional curves. 
	
	\begin{figure}[ht]
		\centering
		\begin{tikzpicture}
			\coordinate[label=above: $E_1$] (E_1) at (0,0);
			\coordinate[label=above: $E_2$] (E_2) at (3,0);
			\coordinate[label=above: $E_3$] (E_3) at (6,0);
			\coordinate[label=below: $A_1$] (A_1) at ($(E_1)!0.5!(E_2)$);
			\coordinate[label=below: $A_2$] (A_2) at ($(E_2)!0.5!(E_3)$);
			\fill (E_1) circle (2pt);
			\fill (E_2) circle (2pt);
			\fill (E_3) circle (2pt);
			\draw[-] (E_1) -- (E_2);
			\draw[-] (E_2) -- (E_3);
		\end{tikzpicture}
		\caption{The Clemens complex of $D_1$.} 
		\label{fig:Clemens-complex}
	\end{figure}
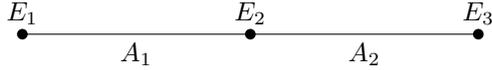
	
	As described in \cite[§ 6]{derenthal_wilsch_21}, and \cite[§ 2.1.6]{wilsch_2021b}, we define $D_A=\bigcap_{E\in A}E$ and $\Delta_{i,A}=D_i-\sum_{E\in A}E$ for a face $A$ of the Clemens complex associated with $D_i$. If $A$ is a maximal face of a Clemens complex, the repeated use of the adjunction isomorphism and a metric on the log-canonical bundle $\omega_{\widetilde{S}}(D_i)$ induce a metric on the bundle $\omega_{D_A}\otimes\mathcal{O}_{\widetilde{S}}(\Delta_{i,A})\vert_{D_A}$ on $D_A$, and hence a Tamagawa measure $\tau_{D_A,\infty}$ on $D_A(\CC)$. We are interested in the modified measure $\norm{1_{\Delta_{i,A}}\vert_{D_A}}_{\mathcal{O}_{\widetilde{S}}(\Delta_{i,A})\vert_{D_A,\infty}}^{-1}\tau_{D_A,\infty}$. For $A$ maximal, the canonical section $1_{\Delta_{i,A}}$ does not have a pole on $D_A$. Since further $D_A(\CC)$ is compact,
	the norm $\norm{1_{\Delta_{i,A}}\vert_{D_A}}_{\mathcal{O}_{\widetilde{S}}(\Delta_{i,A})\vert_{D_A,\infty}}$ is bounded on $D_A(\CC)$ for any metric. Therefore,
	\begin{equation*}
		\norm{\omega\otimes 1_{\Delta_{i,A}}\vert_{D_A}}^{-1}_{\omega_{D_A}\otimes\mathcal{O}_{\widetilde{S}}(\Delta_{i,A})\vert_{D_A,\infty}}\absvalue{\omega}
		= \norm{1_{\Delta_{i,A}}\vert_{D_A}}^{-1}_{\mathcal{O}_{\widetilde{S}}(\Delta_{i,A})\vert_{D_A,\infty}}\tau_{D_A,\infty}
	\end{equation*}
	defines a finite measure on $D_A(\CC)$. Hereby, the above equality is true for any choice of metrics on $\omega_{D_A}$ and $\mathcal{O}_{\widetilde{S}}(\Delta_{i,A})\vert_{D_A}$ that are compatible with the one on their tensor product. The so defined measure is independent of the choice of a form $\omega\in\omega_{D_A}$.
	After normalising this measure with a factor $c_\CC^{\#A} = (2\pi)^{\#A}$, we call it \emph{residue measure} and denote it by $\tau_{i,D_A,\infty}$.
	We refer to \cite[§§ 2.1.12, 4.1]{chambertloir_tschinkel_2010} for details on this construction.
	
	\begin{prop}
		We have
		\begin{equation*}
			\tau_{i,D_A,\infty}(D_{A}(\CC)) = 4\pi^2 
		\end{equation*}
		for every maximal-dimensional face $A$ of the Clemens complex for $D_i$, $i\in\{1,2\}$.
	\end{prop}
	\begin{proof}
		Analogously to Lemma 25 in \cite{derenthal_wilsch_21}, we work in neighbourhoods of the two intersection points $D_{A_1} = E_1\cap E_2$ and $D_{A_2} = E_2\cap E_3$. In order to compute the Tamagawa measures, which are simply real numbers on these points, we consider the charts
		\begin{align*}
			g'\colon \mathbb{A}_K^2&\rightarrow \widetilde{S}, ~ (a,b)\mapsto (a:b:1:1:1:1:1:1:-1-a),\quad\text{and}\\
			g''\colon \mathbb{A}_K^2&\rightarrow \widetilde{S},~ (c,d)\mapsto (1:c:d:1:1:1:1:1:-1-d).
		\end{align*}
		Since $\norm{\mathrm{d}x\wedge\mathrm{d}y} = \norm{\det(J_{f\circ g'})} \norm{\mathrm{d}a\wedge\mathrm{d}b}$, by using \eqref{norm_in_cox_coordinates_1}, we get the norms
		\begin{align*}
			&\norm{(\mathrm{d}a\wedge\mathrm{d}b)\otimes 1_{E_1}\otimes 1_{E_2}\otimes 1_{E_3}}_{\omega_{\widetilde{S}}(D_1)} = \max\lbrace \norm{a^2b^3}_\infty,\norm{a^2b^2}_\infty, \norm{ab^2}_\infty,\norm{ab}_\infty \rbrace,\quad\text{and}\\
			&\norm{(\mathrm{d}c\wedge\mathrm{d}d)\otimes 1_{E_1}\otimes 1_{E_2}\otimes 1_{E_3}}_{\omega_{\widetilde{S}}(D_1)} = \max\lbrace \norm{c^3d^2}_\infty, \norm{c^2d}_\infty, \norm{c^2d^2}_\infty, \norm{cd}_\infty \rbrace.
		\end{align*}
		Then, we obtain the unnormalised Tamagawa
		volume on the points $D_{A_1}(\CC)$ and $D_{A_2}(\CC)$ by 
		\begin{align*}
			\tau'_{1,D_{A_1},\infty} &= \underset{(a,b)\rightarrow (0,0)}{\lim} \frac{\absvalue{ab}^2}{\max\lbrace \absvalue{a^2b^3}^2,\absvalue{a^2b^2}^2, \absvalue{ab^2}^2,\absvalue{ab}^2 \rbrace} = 1, \quad\text{and}\\
			\tau'_{1,D_{A_2},\infty} &= \underset{(c,d)\rightarrow (0,0)}{\lim}
			\frac{\absvalue{cd}^2}{\max\lbrace \absvalue{c^3d^2}^2, \absvalue{c^2d}^2 \absvalue{c^2d^2}^2, \absvalue{cd}^2 \rbrace} = 1.
		\end{align*}
		We renormalise these by multiplying with $c_\CC^2 = 4\pi^2.$ 
		
		In the second case, we additionally consider the intersection point $D_{A_3}=E_3\cap E_4$ and the corresponding chart
		\begin{equation*}
			g'''\colon\mathbb{A}_K^2 \rightarrow \widetilde{S},~ (e,f)\mapsto (1:1:e:f:1:1:1:1:-e-f^2).
		\end{equation*}
		By using \eqref{norm_in_cox_coordinates_2}, we get the norms
		\begin{align*}
			&\norm{(\mathrm{d}a\wedge\mathrm{d}b)\otimes 1_{E_1}\otimes 1_{E_2}\otimes 1_{E_3}\otimes 1_{E_4}}_{\omega_{\widetilde{S}}(D_2)} = \max\lbrace \norm{a^2b^3}_\infty,\norm{a^2b^2}_\infty, \norm{ab}_\infty \rbrace,\\
			&\norm{(\mathrm{d}c\wedge\mathrm{d}d)\otimes 1_{E_1}\otimes 1_{E_2}\otimes 1_{E_3}\otimes 1_{E_4}}_{\omega_{\widetilde{S}}(D_2)} = \max\lbrace \norm{c^3d^2}_\infty, \norm{c^2d}_\infty, \norm{cd}_\infty \rbrace,\quad\text{and}\\
			&\norm{(\mathrm{d}a\wedge\mathrm{d}b)\otimes 1_{E_1}\otimes 1_{E_2}\otimes 1_{E_3}\otimes 1_{E_4}}_{\omega_{\widetilde{S}}(D_2)} = \max\lbrace \norm{e^2f}_\infty,\norm{ef^2}_\infty,\norm{ef}_\infty\rbrace.
		\end{align*}
		As above, we obtain $\tau'_{2,D_{A_1},\infty} = \tau'_{2,D_{A_2},\infty} =1$ and
		\[\tau'_{2,D_{A_3},\infty} = \lim_{(e,f)\rightarrow (0,0)}\frac{\absvalue{ef}^2}{\max\lbrace \absvalue{e^2f}^2,\absvalue{ef^2}^2,\absvalue{ef}^2\rbrace} =1.\]
		We renormalise these Tamagawa volumes by multiplying with $c_\CC^2 = 4\pi^2$.
	\end{proof}
	
	In the next lemma, we compute the rational numbers $\alpha_{i,A}$, which are multiplied with the Tamagawa numbers to compute the archimedean part of the leading constant $c_{i,\infty}$. Here, $A$ is a maximal-dimensional face of the Clemens complex for $D_i$. The factors $\alpha_{i,A}$ were introduced in \cite{chambertloir_tschinkel_2010b} for toric varieties and generalised in \cite[Remark 2.2.9(iv)]{wilsch_2021b} to be 
	\begin{equation*}
		\alpha_{i,A} = \vol\lbrace x\in (\mathrm{Eff}(\widetilde{U}_{i,A}))^\vee\mid \langle x,\omega_{\widetilde{S}}(D_i)^\vee\vert_{\widetilde{U}_{i,A}}\rangle=1\rbrace.
	\end{equation*}
	As in \cite{derenthal_wilsch_21}, in our case the complement of all boundary components not belonging to $A$ is
	\begin{equation}\label{U_tilde_i,A}
		\widetilde{U}_{i,A} = X\setminus\bigcup_{\substack{E_j\subset D_i,\\ E_j\not\in A}}E_j,
	\end{equation}
	and its effective cone is given by $\Lambda_{i,A}=\mathrm{Eff}(\widetilde{U}_{i,A})\subset (\Pic(\widetilde{U}_{i,A}))_\RR$.
	The volume is normalised as in \cite{wilsch_2021b}.
	
	\begin{prop}
		We have
		\begin{align*}
			\alpha_{1,A_1} &= \vol\left\lbrace (t_4,t_5,t_6,t_8)\in\RR^4_{\ge0}~
			\begin{array}{|c}
				-t_4+3t_6+2t_8\ge 1\\
				t_5+2t_6+t_8\le1\\
				t_4+t_6+t_8\le1 
			\end{array}
			\right\rbrace,\\
			\alpha_{1,A_2} &= \vol \left\lbrace (t_4,t_5,t_6,t_8)\in\RR^4_{\ge0}~
			\begin{array}{|c}
				-t_4-t_5+t_6+t_8\ge 0\\
				-t_4+3t_6+2t_8\le 1\\
				t_4+t_6+t_8\le1 
			\end{array}
			\right\rbrace,\\
			\alpha_{2,A_1} &= \vol\left\lbrace(t_5,\dots,t_8)\in\RR^4_{\ge0}~
			\begin{array}{|c}
				3t_6+2t_8\ge1\\
				t_5+2t_6+t_8\le 1\\
				t_6+t_8\le1
			\end{array}
			\right\rbrace,\\
			\alpha_{2,A_2} &= \vol\left\lbrace (t_5,\dots,t_8)\in\RR^4_{\ge0}~
			\begin{array}{|c}
				3t_6+2t_8\le1\\
				-t_5+t_6+t_8\ge0\\
				t_6+t_8\le1
			\end{array}
			\right\rbrace,\quad\text{and}\\
			\alpha_{2,A_3} &= \vol\left\lbrace (t_5,\dots,t_8)\in\RR^4_{\ge0}~
			\begin{array}{|c}
				2t_5+t_6\le1\\
				t_5+2t_6+t_8\ge1\\
				t_6+t_8\le1
			\end{array}
			\right\rbrace.
		\end{align*}
	\end{prop}
	\begin{proof}
		We use the construction and notation from \cite[Proof of Lemma 6]{derenthal_wilsch_21} to compute the~$\alpha_{i,A}$. 
		The data in \cite{derenthal14_lms} shows that $\Pic(\widetilde{S})$ has rank $6$ and is generated by the classes of the negative curves $E_1,\ldots,E_8$, where
		\begin{equation}\label{E_7_and_E_8}
			E_2+E_3-E_4+E_5+E_6-E_7,~\text{ and }~E_1+E_2+E_4+E_5-E_6-E_8
		\end{equation}
		are principal divisors.
		Further, $2E_1+3E_2+2E_3+E_4+2E_5+E_6$ has anticanonical class. 
		For a maximal face $A$ of the Clemens complex, we choose $j_0,j_1\in\lbrace1,\ldots,8\rbrace$ such that $E_{j_0}\in A$, $E_{j_1}\not\in D_i\setminus A$ and such that the classes of $E_j$ for $j\in\lbrace1,\ldots,8\rbrace\setminus\{j_0,j_1\}$ form a basis of $\Pic(\widetilde{S})$. There are (unique) linear combinations $\sum_{j\neq j_0,j_1}a_jE_j$ of class $\omega_{\widetilde{S}}(D_i)^\vee$ as well as $\sum_{j\neq j_0,j_1}b_jE_j$ and $\sum_{j\neq j_0,j_1}c_jE_j$ of the same class as $E_{j_0}$ and $E_{j_1}$, respectively. In our case, by using \eqref{E_7_and_E_8} and the fact that
		\begin{equation}\label{class_omega_S(E_j0)}
			E_4+E_6+E_7+E_8,\quad\text{and}\quad E_6+E_7+E_8 
		\end{equation}
		have class $\omega_{\widetilde{S}}(D_1)^\vee,$ and $\omega_{\widetilde{S}}(D_2)^\vee$, respectively, we can compute the coefficients $a_j,b_j,c_j\in\ZZ$.
		Let $J_i=\lbrace j\in\lbrace1,\ldots,8\rbrace\mid E_j\subset D_i,E_j\not\in A\rbrace$ and $J'_i=\lbrace1,\ldots,8\rbrace \setminus(J\cup\lbrace j_0,j_1\rbrace)$. We have
		\[\Pic(\widetilde{U}_{i,A}) = \Pic(\widetilde{S})/\langle E_j\mid j\in J_i\rangle\]
		by the definition \eqref{U_tilde_i,A} of $\widetilde{U}_{i,A}$.
		Therefore, the classes of $E_j$ for $j\in J'_i$ modulo the classes of $E_j$ for $j\in J_i$ are a basis for $\Pic(\widetilde{U}_{i,A})$, and the classes of $E_j$ for $j\in J'_i\cup\{j_0,j_1\}$ modulo the classes of $E_j$ for $j\in J_i$ yield a basis for the effective cone of $\widetilde{U}_{i,A}$.
		We work with the dual basis of $E_j$. Then, we obtain
		\begin{equation*}
			\alpha_{i,A} = \left\lbrace (t_j)\in\RR^{J'_i}_{\ge0} ~\middle\lvert~ \sum_{j\in J'_i}a_jt_j=1,~\sum_{j\in J'_i}b_jt_j\ge0,~\sum_{j\in J'_i}c_jt_j\ge0\right\rbrace.
		\end{equation*}
		
		For $i=1$, we have to compute the two constants $\alpha_{1,A_j},\,j=1,2$, which are associated with the maximal faces $A_{1}=\{E_1,E_2\}$ and $A_2=\{E_2,E_3\}$ of the Clemens complex for $D_i$. 
		We have to consider the two subvarieties $\widetilde{U}_{1,A_1}=\widetilde{S}\setminus E_3$ and $\widetilde{U}_{1,A_2} = \widetilde{S}\setminus E_1$. For $A_1$, we have $J_1=\{3\}$. We can choose $j_0=1$, $j_1=2$. Then, $J'_1=\{4,\ldots,8\}$. The Picard group of $\widetilde{U}_{1,A_1}$ is $(\Pic(\widetilde{S}))/\langle E_3\rangle$; a basis is given by the classes of $E_4,E_5,E_6,E_7,E_8$ modulo $E_3$, and its effective cone is generated by the classes $E_1,E_2,E_4,E_5,E_6,E_7,E_8$ modulo $E_3$.
		
		By using \eqref{E_7_and_E_8}, we have $[E_1] = [E_3-2E_4+2E_6-E_7+E_8]$ and $[E_2] = [-E_3+E_4-E_5-E_6+E_7]$ in $\Pic(\widetilde{S})$, and $E_4+E_6+E_7+E_8$ has class $\omega_{\widetilde{S}}(D_1)^\vee$ by \eqref{class_omega_S(E_j0)}. Working modulo $E_3$ yields
		\begin{equation*}
			\alpha_{1,A_1} = \vol \left\lbrace  
			(t_4,t_5,t_6,t_7,t_8)\in\RR^5_{\ge0}~
			\begin{array}{|c}
				-2t_4+2t_6-t_7+t_8 \ge 0\\
				t_4-t_5-t_6+t_7\ge 0\\
				t_4+t_6+t_7+t_8=1
			\end{array}
			\right\rbrace.
		\end{equation*}
		We eliminate $t_7$. The last equation yields $t_7=1-t_4-t_6-t_8$. By replacing this in the first two inequalities, we obtain the wanted result.
		
		The computation of $\alpha_{1,A_2}$ is similar. Here, we let $j_0=2,j_1=3$. We use the basis of $\Pic(\widetilde{S})/\langle E_1\rangle$ given by the classes of $E_4,E_5,E_6,E_7,E_8$ modulo $E_1$. The divisor $E_2$ has the same class as \linebreak
		$-E_1-E_4-E_5+E_6+E_8$, and $E_3$ has the same class as $E_1+2E_4-2E_6+E_7-E_8$. As above, $E_4+E_6+E_7+E_8$ has class $\omega_{\widetilde{S}}(D_1)^\vee$. We obtain
		\begin{equation*}
			\alpha_{1,A_2} = \vol \left\lbrace  
			(t_4,t_5,t_6,t_7,t_8)\in\RR^5_{\ge0}~
			\begin{array}{|c}
				-t_4-t_5+t_6+t_8\ge 0\\
				2t_4-2t_6+t_7-t_8\ge0\\
				t_4+t_6+t_7+t_8=1
			\end{array}
			\right\rbrace.
		\end{equation*}
		As above, we eliminate $t_7$ to obtain the result.
		
		In the case $i=2$, we have to compute three constants $\alpha_{2,A_j}$, $j=1,2,3$ associated with the maximal faces $A_1,A_2$ and $A_3=\{E_4,E_1\}$ of the Clemens complex of $D_2$. The three subvarieties appearing in the construction of $\alpha_{2,A_j}$ are $\widetilde{U}_{2,A_1} = \widetilde{S}\setminus\{E_3,E_4\}$, $\widetilde{U}_{2,A_2} = \widetilde{S}\setminus\{E_1,E_4\}$ and $\widetilde{U}_{2,A_3} = \widetilde{S}\setminus\{E_2,E_3\}$. In the first case, we have $J_2=\{3,4\}$ and we choose $j_0=1,j_1=2$. Then, we have $J'_2=\{5,\ldots,8\}$. The Picard group of $\widetilde{U}_{2,A_1}$ is $\Pic(\widetilde{S})/\langle E_3,E_4\rangle$ and a basis is given by the classes of $E_5,\ldots,E_8$ modulo $E_3$ and $E_4$.
		By using \eqref{E_7_and_E_8}, we have $[E_1] = [E_3-2E_4+2E_6-E_7+E_8]$ and $[E_2] = [-E_3+E_4-E_5-E_6+E_7]$ in $\Pic(\widetilde{S})$. Together with \eqref{class_omega_S(E_j0)} we obtain
		\begin{equation*}
			\alpha_{2,A_1} = \vol\left\lbrace (t_5,\dots,t_8)\in\RR^4_{\ge0}~
			\begin{array}{|c}
				2t_6-t_7+t_8\ge 0\\
				-t_5-t_6+t_7\ge0\\
				t_6+t_7+t_8=1
			\end{array}\right\rbrace.
		\end{equation*}
		We eliminate $t_7$.
		The other two cases work similar. For $\alpha_{2,A_2}$ we choose $J_2=\{1,4\}$, $j_0=2$, and $j_1=3$. The Picard group of $\widetilde{U}_{2,A_2}$ is $\Pic(\widetilde{S})/\langle E_1,E_4\rangle$ and a basis is given by the classes of $E_5,\ldots,E_8$ modulo $E_1$ and $E_4$. We have $[E_2] = [-E_1-E_4-E_5+E_6+E_8]$ and $[E_3]=[E_1+2E_4-2E_6+E_7-E_8]$.
		For $\alpha_{2,A_3}$ we choose $J_2=\{2,3\}$ and $j_0=1,j_1=4$. We obtain $\Pic(\widetilde{S})/\langle E_2,E_3\rangle$ for the Picard group of $\widetilde{U}_{2,A_3}$, and a basis is given by the classes of $E_5,\ldots,E_8$ modulo $E_2$ and $E_3$. By using $[E_1] = [-2E_2-E_3-2E_5+E_7+E_8]$ and $[E_4] = [E_2+E_3+E_5+E_6-E_7]$, we obtain the stated results.
	\end{proof}
	
	\begin{cor}
		In total, for $i\in\{1,2\}$ we get the archimedean contribution
		\begin{equation*}
			c_{i,\infty} = \sum_{A}\alpha_{i,A}\tau_{i,A,\infty}(D_{A}(\CC))
			=\constant_i
		\end{equation*}
		to the expected constant, where the sum runs through the maximal faces $A$ of the Clemens complex of $D_i$, with $\constant_i$ defined as in \Cref{def_C_i}.
	\end{cor}
	\begin{proof}
		One easily sees that the two polytopes with volumes $\alpha_{1,A_1}$ and $\alpha_{1,A_2}$ fit together to the one stated in $\constant_1$. The same is true for $i=2$. Using SageMath, we explicitly compute the volume.
	\end{proof}

	\section{Counting over the rational numbers}\label{section: simplifications for the rational numbers}
	From now on, let $K=\QQ$. Analogous to \eqref{def_N_i} we define $N_1(B)$ with $K$ replaced by $\QQ$ and $\ringofintegers$ replaced by $\ZZ$. We do the same for $N_2(B)$.
	Then, the following theorem is the analogue of \Cref{main_theorem}.
	
	\begin{theorem}\label{theorem_over_Q}
		As $B\rightarrow\infty$, we have
		\begin{align*}
			N_1(B) &= \frac{1}{144}\prod_p \left(\left(1-\frac{1}{p}\right)^3\left(1+\frac{3}{p}\right)\right)B(\log B)^4 + O\left(B\log B^3\log\log B\right),\quad\text{and}\\
			N_2(B) &= \frac{11}{72}\prod_p \left(\left(1-\frac{1}{p}\right)^2\left(1+\frac{2}{p}\right)\right)
			B(\log B)^3 + O\left(B\log B^2\log\log B\right),
		\end{align*}
		where the product runs over all primes $p\in\ZZ$.
	\end{theorem}
	\begin{proof}
		This is very similar to the case that $K$ is an imaginary quadratic number field as above.
		
		Similar to \cite{derenthal14_compos}, the parameterisation of integral points on the universal torsor is as in \Cref{cor_correspondence_integral_points_universal_torsor_points}. But here and everywhere below we have $\omega_\QQ = 2$ and $h_K=1$. Hence, the system of integral representatives $\mathcal{C}$ contains only the trivial class $\ringofintegers = \ZZ$, and we obtain $\mathcal{O}_j = \ZZ$ for $j=1,\ldots,9$, $\mathcal{O}_{1*},\ldots,\mathcal{O}_{8*} = \ZZ_{\neq0}$ and $\mathcal{O}_{9*}=\ZZ$. Further, we replace $\norminf{\cdot}$ by the ordinary absolute value $\absvalue{\cdot}$ on $\RR$.
		
		The asymptotic formulas are proved almost exactly as in the imaginary quadratic case. We always have to replace $2/\sqrt{\absvalue{\Delta_K}}$ by $1$, $\pi$ by $2$, complex by real integration and $\sqrt{t_i}$ by $t_i$ in the intermediate results. The main term is computed always analogously, but less technical. The error terms are estimated mostly analogous. The main change is as follows.
		
		For the first summation, we use \cite[Proposition 2.4]{derenthal09} with slightly different height functions and some $\eta_i\in\ZZ^\times = \{\pm1\}$. We compute the error term $2^{\omega(\eta_2)+\omega(\eta_1\eta_2\eta_3\eta_4)}$ as the second summand of the error term in \Cref{lemma_first_summation}.
		The other summations and the completion of the proof of \Cref{theorem_over_Q} by computing $V_0^{(i,j)}(B)$ remain essentially unchanged.
	\end{proof}
	
	\printbibliography
	
\end{document}